\documentclass[12pt]{amsart}

\usepackage{amsmath,amssymb,amscd,amsthm,amsfonts,txfonts}
\usepackage[all,cmtip]{xy} 
\usepackage{cite}
\usepackage{geometry} 
\usepackage[all]{xypic}
\usepackage{graphicx}
\usepackage{mathrsfs}

\newtheorem{thm}{Theorem}[section]
\newtheorem{cor}[thm]{Corollary}
\newtheorem{lem}[thm]{Lemma}
\newtheorem{prop}[thm]{Proposition}
\theoremstyle{definition}
\newtheorem{defn}[thm]{Definition}
\theoremstyle{remark}
\newtheorem{rem}[thm]{Remark}
\numberwithin{equation}{section}

\newcommand{\C}{\mathbb{C}} 
\newcommand{\R}{\mathbb{R}}

\DeclareMathOperator{\Sym}{Sym}

\DeclareMathOperator{\tr}{tr}

\DeclareMathOperator{\F}{\mathcal{F}}

\DeclareMathOperator{\M}{\mathcal{M}}

\DeclareMathOperator{\vol}{vol}
\DeclareMathOperator{\D}{\mathcal{D}}
\DeclareMathOperator{\E}{\mathcal{E}}
\DeclareMathOperator{\J}{\mathcal{J}}
\DeclareMathOperator{\Ric}{Ric}

\providecommand{\abs}[1]{\lvert#1\rvert}

\newcommand{\dist}[1]{\mathrm{dist}(#1)}
\newcommand{\diam}[1]{\mathrm{diam}(#1)}

\newcommand{\I}{\mathcal{I}}
\newcommand{\db}{\mathscr{D}}
\newcommand{\Mfd}{\mathscr{M}}
\newcommand{\tc}{\mathcal{H}}
\newcommand{\st}{\mathcal{S}}
\newcommand{\Har}{\mathscr{H}}

\title[Measure estimate and Harnack inequality]{Measure Estimates, 
Harnack Inequalities and Ricci Lower Bound}

\author{Yu Wang}
\address{Yu Wang\\
Department of Mathematics\\
Columbia University, NY, U.S.} \email{yw2340@math.columbia.edu  }

\author{Xiangwen Zhang}
\address{Xiangwen Zhang\\
Department of Mathematics and Statistics\\
McGill University, Canada} \email{xzhang@math.mcgill.ca}

\begin{document}

\begin{abstract}
Consider a Riemannian metric-measure space. We establish an Alexandrov-Bakelman-Pucci type estimate connecting the Bakry-\'Emery Ricci curvature lower bound, the modified Laplacian and the measure of certain special sets. We apply this estimate to prove Harnack inequalities for the modified Laplacian operator (and fully non-linear operators, see the Appendix). These inequalities seems not available in the literature; And our proof, solely based on the ABP estimate, does not involve any Sobolev inequalities nor gradient estimates. We also propose a question regarding the characterization of Ricci lower bound via the Harnack inequality.
\end{abstract}

\maketitle


\section{Introduction}
This paper is devoted to generalize Alexandrov-Bakelman-Pucci (abbrev. ABP) techniques to general Riemannian setting and use them to study the relation between Ricci lower bound and elliptic PDEs on Riemannian metric-measure spaces. In particular, we establish an ABP-type inequality (Thm.\ref{Measure Estimate}), which connects the measure of some specific sets (contact sets, Defn.\ref{contact set}) and  $N$-Baker-\'Emery (abbrev. BE) Ricci curvature ($N \in [n, \infty]$). The idea of this work is largely influenced by the remarkable paper of Cabr\'e (\cite{Cab}).

To illustrate the power of the ABP-techniques, we shall consider, on a smooth Riemannian metric-measure space $(\Mfd, g, \nu )$ with $\nu =e^{-V} \vol_{g}$, the modified Laplacian operator
\[
\Delta_{\nu} u = \Delta u - g (\nabla u, \nabla V) \Leftrightarrow L_{\nabla u} \nu = (\Delta_{\nu} u) \nu,
\]
($L$ stands for the Lie derivative). We shall prove the Harnack inequalities (Thm.\ref{Harnack sup finite}--Thm.\ref{Harnack soln finite}) for this operator under assumption of local lower bound of $N$-BE Ricci curvature ($N < \infty$). Harnack inequalities in this generality seem to be unavailable in the literature; and our proof differs completely from standard methods in geometric analysis. In particular, it does not involve Sobolev inequalities nor gradient estimates. This proof also applies for fully-nonlinear operators (see the Appendix) and it suggests us to consider characterizing Ricci lower bound by the Harnack inequality.

ABP-techniques are of central importance in the modern study of the elliptic equations. They have been widely applied in the study of various classes of linear and nonlinear elliptic equations in the Euclidean space (see \cite{CC} and reference therein). However, they seem not much recognized in the field of geometric analysis and Riemannian geometry. Part of the reasons that limit the application of these techniques in Riemannian geometry are the following. 

Fix a given function $u$ (solution to some PDE), a key idea in ABP techniques is to consider the set of minimum points of $u - l$ for each possible linear function $l$, and consider its image under the gradient mapping $\nabla u$. While non-constant linear functions do not in general exists on a Riemannian manifold; and $\nabla u$, whose image lies in the tangent bundle, seems difficult to deal with.

These difficulties are resolved by Cabr\'e in \cite{Cab}. In this pioneer work, Cabr\'e proposed to replace linear functions by paraboloids--squared distance functions $\rho^2 (\cdot, y)$-- and consider the following special sets:
\begin{defn}
\label{contact set}
Let $\Omega$ be a bounded subdomain of the smooth Riemannian manifold $(\Mfd, g)$ and $u \in C(\Omega)$. For a given $a \geq 0$ and $E$ compact subset of $\Omega$, the contact set of opening $a$ is defined by 
\[
A(a, E / \Omega, u) := \{x \in \overline{\Omega} , \;  \inf_{\overline{\Omega}} \{  u + a \rho^2 ( \cdot, y)\} = u (x) + a \rho^2 (x, y)   \text{ for some } y\in E  \},
\]
where $\rho$ is the distance function of the metric $g$. 
\end{defn}

And replace the gradient map $\nabla u$ by the map 
\begin{equation}
\label{transportation}
F[u] (x) := \exp_{x} (\nabla u  (x)) ,\quad   u \in C^2 (M).
\end{equation}
Based on this idea, Cabr\'e was able to control the integral of the Laplacian $\Delta u$ (or more generally non-divergence linear operator, see Defn.1.1 in \cite{Cab}) over a sub-level set from below by the volume of the domain (see Lemma 4.1 in \cite{Cab}). Then following the approach of Krylov-Safanov (\cite{KS1}, \cite{KS2}), Harnack inequalities (for non-divergence equation) on spaces with non-negative sectional curvature are derived from this estimate via Calder\'on-Zygmund decompositions. Recently, the approach in \cite{Cab} has been extended by S. Kim (\cite{Kim}), who replaced the assumption of non-negative sectional curvature by certain balanced condition on sectional curvatures according to the given operator $L$ ( see condition $4$ and $5$  in Sec.1 of \cite{Kim}). In particular, for the case of Laplacian, that condition is equivalent to Ricci nonnegative.

Nevertheless, power of Cabr\'e's approach has not yet been fully explored. Following his approach, combined with some recent development in the theory of optimal transport (\cite{Villani} and reference therein), we can extend the ABP techniques with a considerablely larger generality -- only local BE-Ricci bound required (see. Defn.\ref{BE Ricci} for the BE-Ricci curvatures). 
In particular, we prove the following \textit{Measure Estimate Formula} , which resembles the Euclidean version of ABP estimate.

\begin{thm}
\label{Measure Estimate}
Let $(\Mfd, g, \nu)$ be a complete Riemannian
metric-measure space with dimension $n \geq 2$.
Let $E$ be a closed subset of a geodesic ball $B_{r}$ and $ u \in C(\overline{B}_{r})$; Let $K \geq 0$ and $N \in [n, \infty]$ be two constants. 

Suppose $a>0$ and $A(a, E/B_{r}, u) \subset B_{r}$; Suppose there exists a subdomain $\Omega'$ containing $A(a, E/B_{r}, u)$ such that $u \in C^2(\Omega')$. Then the following statements holds:

\smallskip
If $N < \infty$, denote $\omega_{K,N} = 2\sqrt{K/N}$, then
\[
\Ric_{N, \nu}|_{B_{r}} \geq -Kg \Rightarrow  \nu [E]\leq
\int_{A (a,  E/B_r , u ) } \Big\{ \D_{K, N, R}[u/a] (x) \Big\}^N\;  \nu (dx) 
\]
where for any $x \in A (a, E/\Omega,u)$,
\[\begin{split}
\D_{K,N, r} [ u/a] (x) & := \frac{\sinh(  r\omega_{K, N})}{ r\omega_{K, N} } \biggl[  r\omega_{K, N}\coth(r\omega_{K, N} )   + \frac{\Delta_{\nu}u (x)}{ N a}  \biggr] 
\end{split}\]
In the case $K= 0$, expressions are understood as their obvious limits.  

\smallskip

If $N = \infty$
\[
\Ric_{\infty, \nu}|_{B_{r}} \geq - K g \Rightarrow \nu [E] \leq 
\int_{A (a , E/B_{r} , u)}\exp \Big\{ \D_{K, N,  r}[u/a] (x)    \Big\}\;  \nu (dx) 
\]
where
\[
\D_{K,N, r}[u/a] (x) :=  2r^2 K  + \frac{\Delta_{\nu}u  (x)}{ a}.
\]
\end{thm}

\begin{rem}
The upshot of the above formula is that the integration is done only on a special set--the contact set, and the lower bound of this integral can be controlled. This is the essence of the Euclidean ABP estimate. Unlike the ABP estimate in Euclidean space, the above formula does not involve infima of the unknown function. However, we shall see this would not limit its application.
\end{rem}

\begin{rem}
We have assumed the underlying manifold is smooth and $u$ is $C^2$ near $A(a, E/\Omega, u)$. In fact, we only need $u$ to be semi-concave (see, Defn.16.4, p.429 in \cite{Villani}) near $A$. Moreover the theorem can be established on Alexandrov spaces. However, to avoid heavy formulations and to better present the main ideas, we shall stay with $C^2$-functions. Our proof can be easily adopted to including semi-concave functions and Alexandrov spaces.
\end{rem}

\begin{rem}
Refer to the proof of Thm.\ref{fully nonlinear} in the Appendix for generality and sufficiency of only considering $\Delta_{\nu} u$ in the above formula. Indeed, on the contact set other linear (appropriate nonlinear) operators can be controlled from below by the Laplacian. 
\end{rem}

We want to remark that the Measure Estimate formula is valid in all effective dimension, including particularly the case $N = \infty$. We believe this formula will have applications in many geometric problems. We shall discuss some of these applications in a separate paper.

The underline idea of proving the above theorem is indeed contained in \cite{Cab}. That is, apply the Area formula to the map $F [u]$ (Eq.\ref{transportation}) on contact set $A$. Here, rather than use the direct calculation of Jacobi determinant of $F[u]$ given in \cite{Cab}, we employ an ODE comparison estimate suggested in Ch.14 of \cite{Villani}. Besides allowing us to establish the estimate for very general curvature condition, this ODE estimate matches in a remarkable way the fine structures of contact sets (see. Lem.\ref{cut-focal free}, Lem.\ref{geometric Jacobi}).

\smallskip

Similar to that in \cite{Cab} (and \cite{Kim}), the Krylov-Safanov method allows one to deduce the following Harnack inequalities from Thm.\ref{Measure Estimate}.

Given $K \geq 0, N \in [n, \infty)$, denote
\begin{equation}
\label{eta}
\eta =   \eta_{K,N, 2R} =  1 + 8 R\log 2\sqrt{K/N}
\end{equation}
in the following statements. All integrals are preformed against $\nu$. The manifold $\Mfd$ has dimension $n \geq 2$.

\begin{thm}
\label{Harnack sup finite}
Let $(\Mfd, g, \nu)$ be a complete smooth Riemannian metric-measure space. Let $K \geq 0$ and $N \in [n, \infty)$. Let $u \in C^2 (B_{2R}) \cap C(\overline{B}_{2R})$ and $f \in C (B_{2R})$. 
Suppose
\[
\Ric_{N, \nu} |_{B_{R}} \geq - K g \quad 
\Delta_{\nu} [u] \leq  f \text{ in } B_{2R}, \quad u \geq 0 \text{ in } B_{2R}
\]
Then,
\begin{equation}
\label{sup inequality}
 (\fint_{B_{R/2} } u^{p_0})^{1/p_0} \leq C_0 \biggl\{ \inf_{B_{R/2}}
 u 
  +  R^2\biggl(\fint_{B_{2R}} \abs{f}^{N \eta} \biggr)^{1/(N\eta)} \biggr\}
\end{equation}
where $p_0, C_0 $ are constants only depending on $\sqrt{K}R, N$. Morevoer $C_0 = e^{2/p_0}$.
\end{thm}

\begin{thm}
\label{Harnack sub finite}
Let $(\Mfd, g, \nu)$ be a complete smooth Riemannian metric-measure space. Let $K \geq 0$ and $N \in [n, \infty)$. Let $u \in C^2 (B_{2R}) \cap C(\overline{B}_{2R})$ and $f \in C (B_{2R})$

Suppose
\[
\Ric_{N, \nu} |_{B_{R}} \geq - K g \quad  \;
\Delta_{\nu} [u] \geq  f \text{ in } B_{2R}.
\]
Then, for any $p >0$
\begin{equation}
\label{sub inequality}
\sup_{B_{R/2}} u \leq  C_2 (p) \biggl\{ \biggl(\fint_{B_{R}} (u^+)^p \biggr)^{1/p}  
 +  R^2\biggl(\fint_{B_{2R}} \abs{f}^{N \eta} \biggr)^{1/(N\eta)} \biggr\}
\end{equation}
where $C_{1} (p)$ is a constant only depending on $\sqrt{K}R, N$ and $p$.
\end{thm}

\begin{thm}
\label{Harnack soln finite}
Let $(\Mfd, g, \nu)$ be a Complete Smooth Riemannian Metric-Measure Space. Let $K \geq 0$ and $N \in [n, \infty)$. Let $u \in C^2 (B_{2R}) \cap C(\overline{B}_{2R})$ and $f \in C (B_{2R})$

Suppose
\[
\Ric_{N, \nu} |_{B_{R}} \geq - K g \quad  \;
\Delta_{\nu} [u] =  f \text{ in } B_{2R}, \quad u \geq 0 \text{ in } B_{2R}
\]
Then
\begin{equation}
\label{soln inequality}
\sup_{B_{R/2}} u \leq  C_2 \biggl\{ \inf_{B_{R/2}} u 
  +  R^2\biggl(\fint_{B_{2R}} \abs{f}^{N \eta} \biggr)^{1/(N\eta)} \biggr\}
\end{equation}
where $C_{2}$ is a constant only depending on $\sqrt{K}R$ and $N$.
\end{thm}

\begin{rem}
In the case $K = 0$, the integral expressions of the right-hand side $f$ reduce to the standard averaged $L^{N}$-norm. This agree with the Harnack inequalities in \cite{Cab} when $K =0$ and $\nu =\vol_{g}$. Increasing the exponent of integration by a factor depending on Ricci curvature lower bound is necessary. This can be seen easily from the examples in $(K,N)$-Hyperbolic space. 
\end{rem}

\begin{rem}
For the readers' convenience, a set of explicit estimate of $p_0, C_1 (p_0), C_2$ shall be given at the end of \S \ref{notations} (see Lem.\ref{constants estimate}). For general $p$, $C_1 (p)$ is obtained from $C_1 (p_0)$ by interpolation (see. the proof in \S \ref{soln}).
\end{rem}

\begin{rem}
We emphases that constants ($p_0, C_1, C_2, \eta$) in above theorems depend on the product $\sqrt{K}R$, rather on $R$ or $K$ alone. In particular, if $K =0$, $p_0, C_1,$ and $C_2$ are independent of radius $R$ and $\eta= 1$. Hence the above theorems recover the Harnack inequalities (for the Laplacian) in \cite{Cab} and \cite{Kim} (see the Appendix for case of fully-nonlinear uniform elliptic equations). 
\end{rem}

\begin{rem}
Higher regularity estimates, including gradient and $C^{2}$ estimates, can be obtained from the Harnack inequalities via the methods given in \cite{CC}.
\end{rem}

Besides proving Harnack inequalities in larger generality, we provide a different presentation of Krylov-Safanov argument in proving Harnack inequalities. In this presentation, the Calder\'on-Zygmund decomposition (used in \cite{Cab}) is replaced by Vitali's covering lemma. Though it follows essentially same spirits as that in \cite{Cab}, the argument here seems more elementary and transparent. It is similar to the covering argument used by L.A. Caffarelli in his breaking through work on real Monge-Amp\'ere equations (\cite{Caf3}).  We learned this argument from O. Savin's lecture at Columbia University and his paper \cite{Savin}.

Harnack inequalities has been intensively studied in geometric analysis. An incomplete list includes: the remarkable work of S.Y. Cheng and S.-T Yau \cite{CY}, who considered $\Delta u = - \lambda u, \lambda >0 $ and proved Harnack inequality for solutions via establishing a sharp gradient estimate; Later, J. Cheeger, M. Gromov, M.Taylor, S.Y. Cheng, P. Li and S.-T. Yau has employed the method of De Georgi-Nash-Moser iteration and C.B. Croke's work on Sobolev inequalities (\cite{Croke}) to consider differential inequalities (\cite{CGT}, \cite{CLY}). The optimal results via this approach is given by L. Saloff-Coste (\cite{SC}). Based on the penetrating work on Sobolev inequalities due to N. Varopoulos (\cite{Var}) and Cheng and Yau's gradient estimate, he proved the Harnack inequality for divergent operators on manifolds with standard Ricci curvature bounded from below. These works have also been extended to various general cases. For example, in \cite{Zhang}, \cite{WangZ}, t!
 hese authors studied the gradient estimate for $p-$harmonic function on Riemannian manifolds. Recently, Li \cite{LiXD} (see also \cite{Br}) followed the main line of \cite{CY} and gave the Harnack inequality for \textit{solutions} of the modified Laplacian equation on Riemannian Metric-Measure space. 

Comparing to the above mentioned work, 
besides the generality of our results,  we would like to emphases that our approach, following Cabr\'e, differs completely from above mentioned work.

\smallskip


The statement of the Harnack inequality (Thm.\ref{Harnack soln finite}) and the key ingredient (Thm.\ref{Measure Estimate}) in our proof suggests us that constants in Harnack inequalities are of geometric meaning. In fact, we believe that these constants characterize the lower bound of Ricci curvature (for Riemannian metric-measure space, also related to the effective dimension). Based on this point, we propose some questions concerning the relation between constants in the Harnack inequality and the lower bound of Ricci curvature. In the present paper, we shall provide a precise formulation of these questions in this paper and suggest some affirmative evidence. We think this kind of characterization would have applications in the study of Gromov-Hausdorff convergence, geometric flows and Alexandrov spaces. Our idea here is largely inspired by the work of J. Lott, C. Villani (\cite{LV1}, \cite{LV2}).

\smallskip
The paper is organized as follows: In \S2, we fix our notations and conventions. In particular, we give a full list of constants involved in the later proof. \S3 and \S4 are devoted to study the contact sets and the Jacobi determinant of $dF[u]$ respectively. In these two section, we shall see how the contact sets, the Jacobi fields and the underlying geometry interact with each other. In \S5, we prove the Measure Estimate Formula --Thm.\ref{Measure Estimate}. \S6 and \S7 contains some preparations for the proof of Harnack inequalities (Thm.\ref{Harnack sup finite}--Thm.\ref{Harnack soln finite}). \S8 contains the main technical lemma in proving these theorems. In \S9 and \S10, Thm.\ref{Harnack sup finite}--Thm.\ref{Harnack soln finite} are proved. In \S11, we discuss a possible way to characterize the Ricci lower bound by Harnack inequalities. In the Appendix, we extend the method in this paper to prove the Harnack inequalities for fully-nonlinear uniform elliptic operators!
  on Riemannian manifolds.

\bigskip

\section{Notations, Conventions and Constants}
\label{notations}
In order to avoid any potential confusion, we first state our convention regarding the curvatures and the cut-locus.

$\bullet$ \textit{Riemannian Metric-Measure space:} In the paper, the background manifold is the Riemannian metric-measure space $(\Mfd, g, \nu)$ where $g$ is the Riemannian metric on $\Mfd$ and $\nu=e^{-V}\vol$ is a reference measure with $V: M\rightarrow \mathbb{R}$ a $C^2$ function. Notice that, if $V=0$, then $\nu$ is just the usual volume measure $\vol_g$.

\smallskip

$\bullet$ \textit{Curvatures:} Recall the definition of Riemannian curvature tensor 
\begin{equation}
\label{defn Riem}
\mathrm{Riem} (X, Y) := D_{Y} D_{X} - D_{X} D_{Y} + D_{[X,Y]}, \quad X, Y \in T\Mfd
\end{equation}
and that of Ricci curvature
\begin{equation}
\label{defn Ric}
\Ric_x (Z) := \sum_{i} g_x ( \mathrm{Riem} (Z, e_i) Z, e_i ), \quad Z \in T_x\Mfd
\end{equation}
where $e_i$ is the orthogonal basis w.r.t. $g$.
Note here our convention on $\Ric$ is standard, and it agrees with both the reference \cite{Villani} and \cite{Peter}; our convention of $\mathrm{Riem}$ agrees with the reference \cite{Villani} (p.371) but differs from \cite{Peter}(p.33) a sign.

We also recall the following definition:
\begin{defn}
\label{BE Ricci}
The Bakry-Emery (abbrev. BE) Ricci curvature associate to $\nu$ with effective dimension $N \in [n, \infty]$ is defined by
\[
\Ric_{N, \nu} = \begin{cases}
 \Ric & N = n,   \\
\Ric + D^2 V  - \frac{D V \otimes D V}{N - n} &  N > n \\
\Ric + D^2 V & N = \infty.
\end{cases}
\]
Here, we assume $V = 0$ whenever $N = n$. 
\end{defn}

\begin{rem}
For the importance and the geometry of the Bakry-Emery Ricci curvature, one may refer to \cite{Lott}, \cite{Villani}, the recent work of \cite{WW} and the reference therein. In particular, $\Ric_{\infty, \nu}$ plays an important role in Pereleman's work (\cite{Perelman}) on Hamilton's Ricci flow. .
\end{rem}

\smallskip

$\bullet$ \textit{Cut-locus}: We recall the definition of cut-points and focal points. We follow the convention in \cite{Villani} (p.193). Note, this convention may differ from some text, but will not affect the generality of this paper.

\begin{defn}
Fix $x \in (\Mfd ,g)$, a point $y $ is called a cut point of $x$ if there is a geodesic $\gamma (t)$ such that $\gamma (0) = x$ and $\gamma (t_{c}) = y$ and satisfies that i) $\gamma (t)$ is minimizing for all $t \in [0, t_{c})$ and ii) $\gamma (t_{c} + \epsilon )$ is not minimizing for any $\epsilon > 0$.
\end{defn}

\begin{defn}
Two points $x$ and $y$ are said to be focal (or conjugate) if $y$ can be written as $\exp_{x} (t W), W\in T_{x} \Mfd$, and the differential $d |_{W} \exp_{x} (t \cdot )$ is not
invertible.
\end{defn}

\begin{defn}
Given a point $x \in \Mfd$, the cut-locus $Cut (x)$ of $x$ is the set consisting of all cut-points and focal (conjugate) points of $x$.
\end{defn}

\begin{rem}
Being cut-points and focal points are symmetric relations. $x \in Cut (y)$ if and only if $y \in Cut (x)$.
\end{rem}

\smallskip

$\bullet$ \textit{Contact Relations:} We recall the following terminologies:
\begin{defn}
\label{touch}
Let $\Omega $ be a subdomain of $\Mfd$. Let $u$ and $\varphi$ be two continuous functions in $\Omega$. 

Let $x_0 \in \Omega$ and $U$ be a subset (%
not necessarily open) of $\Omega$, we say $\varphi$ touches $u$ from above (resp. below) at $x_0$ in $U $ if $\varphi (x)\geq w(x)$ (resp. $\varphi (x) \leq w (x)$) for all $x\in U$ and $\varphi (x_0) = w(x_0)$. 

We say $\varphi$ touches $u$ from above (resp. below) at $x_0$ if there is a neighborhood $U$ of $x_0$ such that $\varphi$ touches $u$ from above (resp. below) at $x_0$ in $U$. 

\end{defn}

\medskip

$\bullet$ \textit{Convention in Notations} We also have the following conventions in notations. 

$i)$ Throughout this paper, a later $C$, \textit{Without any Subscript}, represents a \textit{Pure} constants \textit{Greater} than $1$. It might change from line to line. However we emphases that it does \textit{Not} depends on any parameter. Moreover, to make the proof more transparent, we shall try to minimize the usage of $C$ and try to be explicit.
\smallskip

$ii)$ We always denote the standard dimension of $\Mfd$ by $n$, and we assume throughout the paper that $n \geq 2$. 
\smallskip

$iii)$ We always denote $B_{r} (x)$ to be the geodesic ball of radius $r$ centered at $x$. We shall omit the center $x$ when it has no particular importance and also cause no confusion.
\smallskip

$iv)$ Throughout the paper, integrals are performed against the measure $\nu$; and the distance function is denoted by $\rho$. The notation $\rho_{y}$ means the distance from a fixed point $y$.

\medskip

$\bullet$ \textit{Special Functions and Notations} In the rest of this paper, in particular the proof of our main theorems, many parameters and functions get involved. In order to give a clear presentation, we shall make several short-hand notions. Here, we list these notations and some basic facts regarding them. 

$i)$ Let $K \geq 0,  N \in [n, \infty)$ and $r > 0$, we denote
\begin{equation}
\label{C omega, db}
\omega_{K, N}:=2 \sqrt{K/N}, \quad \db_{K,N,r}:= 2^{N} e^{4r\sqrt{NK}}, \quad 
\end{equation}
and define also
\begin{equation}
\label{C eta}
\eta_{K,N,r}:= \log \db_{K,N,r}/(N \log 2)= 1+ 4  r\sqrt{NK}\log 2
\end{equation}
\textbf{Notice}: The subscript $K,N,r$ does \textit{Not} mean $\db_{N,K,r}$ or $\eta_{K,N,r}$ depends on $K$ nor $r$ in a separated way. It is just for convenience. When no confusion arise, one or all subscripts might be omitted. In particular, we shall often use $\db_{r}, \eta_{r}$ in replacing $\db_{K,N,r}$ and $\eta_{K,N,r}$. However, this is only for convenience. Again, it does not mean $\eta$ depends on $r$ nor $K$ in a separated way.

\smallskip

$ii)$ Let $t \in [0, \infty)$, define
\begin{equation}
\label{tc, st}
\tc (t) := t \coth (t), \quad \st (t) := \sinh (t ) /t
\end{equation}
Note $\tc, \st$ are differentiable and have positive derivative for all $t >0$; $\tc (0) = \sc (0) = 1$ by limit. Moreover, $\st (t) \cdot \tc (t) = \cosh (t)$. Another useful observation here is that
\begin{equation}
\label{linear tc}
\tc (t) \leq 1 + t, \quad t \geq 0
\end{equation}
\smallskip

$iii)$ Let $q \geq 1$ be a constant and $f$ be a continuous function, we denote
\begin{equation}
\label{KN integral}
\I_{K,N} (f, B_{R}, q) := r^2 \biggl( \fint_{B_{r} }  \abs{f (x)}^{N q} \; \nu (dx) \biggr)^{1/(N q)}
\end{equation}
Properties for this integral are given in \S \ref{integral}.

\medskip

$\bullet$ \textit{Constants in Proofs}. The following constants shall be used frequently in the proof of Harnack inequalities (\S \ref{key growth lemma} -- \S\ref{soln}). They are \textit{Fixed} for the entire paper. We also provide some rough estimates for them. 

\begin{rem}\textbf{Notices} The scenario making use the constants below is the following. Fix a large ball $B_{R}$, we shall perform some estimates on a small ball $B_{r} (x) \subset B_{R}$. All the constants shall depends on the $\sqrt{K}R$. But they does \textit{Not} depend on size of the small ball. We \textit{emphases} that all the constants below depend \textit{Only} on $\sqrt{K}R$ and $N$. They do \textit{Not} depend on $R$ \textit{Nor} $K$ in a separated way.
\end{rem}

Denote
\begin{equation}
\label{C alpha}
\alpha := N \tc(\omega_{K,N} R) \leq  2R\sqrt{NK}  + 1
\end{equation}
Here we used Eq.(\ref{linear tc}).

\begin{equation}
\label{C mu}
\mu  :=\biggl[(18)^3\;\alpha^{2}(18)^{\alpha}\cosh(\omega R) \biggl]^{-N} \db_{4R}^{-4}  \geq e^{ -C\{ R\sqrt{KN} + N^2\} }
\end{equation}
Here, we have applied the estimate of $\alpha$

\begin{equation}
\label{C M} 
M := 2\alpha^2 (18)^{\alpha} \leq e^{  C( R\sqrt{KN} + N ) }
\end{equation}
Again, we have employed the estimate of $\alpha$.
\smallskip

\begin{equation}
\label{C delta}
\delta_0 := \biggl( 2\D_{2R}^{4/N} S( \omega r)  \biggr)^{-1} \geq e^{- C (R\sqrt{KN}  + 1)}
\end{equation}

\begin{equation}
\label{C C3}
C_3 :=2\db_{2R} ( M^{1/p_0}/ \mu  )^{1/N}, 
\end{equation}
Note, since $\eta > 1$

\begin{equation}
\label{C3 estimate}
\db_{2R}\frac{M}{\mu^{1/p_0}} ( \frac{1}{C_3})^{N\eta/ p_0} < 1
\end{equation}

\begin{equation}
\label{C pr}
p_1 := p_0 /(N\eta)
\end{equation}

The last few constants are those appear in the statement of Thm.\ref{Harnack sup finite} -Thm.\ref{Harnack soln finite}. We shall give explict forms of $p_0, C_0, C_1 (p_0), C_2$ here and in later proof, it shall be clear that these choices are sufficient. We shall also give some rough estimat of these constants. 

\begin{lem}
\label{constants estimate}
Define $p_0, C_0, C_1 (p_0), C_2$ as follows
\begin{equation}
\label{C p0C0}
p_0 =\frac{ 1 - \log[1 + (e - 1)(1 - \mu)] }{\log M}, \quad C_0 := e^{2/p_0}
\end{equation}
and
\begin{equation}
\label{C C1p0}
C_2 = C_1 (p_0):=\biggl( 3 C_3 \sum_{k=0}^{\infty} \frac{1}{(1 + 1/M)^{k p_1}}  \biggr)^{1/p_1}\frac{1}{\delta_0}
\end{equation}

The following statement holds:

i) $1 + (M^{p_0} -1) \sum_{k=0}^{\infty} (M^{p_0} (1 - \mu))^{k}= e$;

ii) $e^{1/p_0} \geq \delta_0 \geq 1$;  

iii) The constant $p_0, C_0$ satisfies
\[
p_0 \geq \frac{\mu}{4\log M} \geq e^{- C [R\sqrt{KN} + N^2]}, \quad C_0 \leq \exp[e^{C [R\sqrt{KN} + N^2]}];
\]

iv) The constant $C_2  =C_1 (p_0)$ satisfies
\[
C_2 = C_1 (p_0) \leq \exp [e^{C (R\sqrt{KN} + N^2)}]  ;
\]
\end{lem}

\begin{proof}
i) can be verified directly.

\smallskip

To show ii), note the relation
\begin{equation}
\label{log relation}
\frac{x}{e}  \leq 1 - \log [e- x] \leq \frac{x}{e-1}
\end{equation}
we have then
\[
\frac{1}{p_0} \geq \log M /\mu.
\]
By the choice of $M, \mu, \delta_0$, it is clear that
\begin{equation}
e^{1/p_0} \geq \frac{1}{\delta_0}.
\end{equation}

\smallskip

To see iii), one just need to note $(1 + 1/M) \geq (1 + 1/M)^{p_1}$ for $p_1 < 1$.

\smallskip
By (Eq.\ref{log relation}), it is then easy to show
\[
p_0 \geq \frac{\mu}{4\log M} \geq e^{- CN^2 [R\sqrt{K/N} + 1]}
\]
The estimate of $ e^{1/p_0}$ then follows. 

\smallskip

Finally to show v),consider the following manipulation. Denote $b:=((M+1)/M)^{p_1}$
\[\begin{split}
\sum_{k=0}^{\infty} \biggl( \frac{M}{M+1}\biggr)^{p_1 k}& = \frac{1}{1 - e^{-\log b}} = \frac{1}{2} + \frac{1}{2} \frac{1 + e^{-\log b}}{1 - e^{-\log b}} \\
&  = \frac{1}{2} + \frac{1}{2} \coth[\frac{1}{2} \log (b)] = \frac{1}{2} + (\log b)^{-1}\tc(\frac{1}{2} \log (b)) \\
& \leq 1 + \frac{1}{\log(b)} \leq \frac{3 M }{p_1}.
\end{split}\]
In the last two inequalities, we have used (Eq.\ref{linear tc}) and an estimat of  $\log(1+1/M)$ similar to (Eq.\ref{log relation}). Then, we can then easily estimate $C_2 = C_1 (p_0)$ as stated.
\end{proof}

\bigskip

\section{Contact Set}

In this section, we investigate properties of contact sets (recall Defn.\ref{contact set}), in particular the behavior of the unknown function $u$ on its associated contact sets. We shall see the contact sets recognize the underline metric geometry in an elegant way.

First, we state an alternative characterization of contact set, which has mentioned in \cite{Cab}.
\begin{defn}
\label{parabolid}
The concave parabolid $P_{a,y}$ of vertex $y$ and opening $a$ is a function of the form
\[
P_{a, y} := -\frac{a}{2} \rho^2 (x, y) + c_y, \quad c_y, a \in \R, \; a \geq 0
\]
Similarly, one defines convex parabolid.
\end{defn}

\begin{prop}
Let $ u \in C(\overline{\Omega})$, $E \subset \Omega$ closed and $a \geq 0$. Then $x \in A (a, E, u)$ if and only if there exists a concave paraboloid $P_{a,y}$ of opening $a$ and vertex $y \in E$ that touches $u$ in $\overline{\Omega}$.
\end{prop}

\begin{proof}
Immediately follows from the definitions (Defn.\ref{contact set}, Defn.\ref{touch} and Defn.\ref{parabolid}). 
\end{proof}

The following proposition contains some basic properties of contact sets (see. \cite{Savin}). Its proof is a routine check up and hence be omitted.
\begin{lem}
\label{basic contact}
Let $u \in C (\overline{\Omega})$ and $E \subset \Omega$ closed, then 

a) For all $a\geq 0$, $A (a ,E , u)$ is closed (hence $\nu$-measurable).

b) If $u_k \rightarrow u$ uniformly in $\Omega $, then
\[
\limsup_{k \rightarrow \infty} A (a, E/\Omega,  u_k) = \bigcap_{j =1}^{\infty} \bigcup_{k \geq j} A_{k} \subset A (a ,E/\Omega ,u).
\] 

c) if $a_{k} \rightarrow 0$, then
 \[
\limsup_{k \rightarrow \infty} A (a_{k}, E/\Omega, u) = \bigcap_{j =1}^{\infty} \bigcup_{k \geq j} A_{k} \subset A (0 ,E/\Omega ,u).
\] 

d) if $E \subset F$, then
\[
A (a, E/\Omega, u) \subset A (a, F/\Omega, u).
\]
\end{lem}

\begin{rem}
Though we shall only need $a)$ of the previous lemma in this paper, the other three properties are of practical value. We shall illustrate some of their applications in a separated paper.
\end{rem}

\smallskip

To consider the interaction between contact sets and the underline metric geometry, we shall need the following notion and proposition from standard Riemannian geometry. 

We first recall the Hessian bound in support sense introduced by Calabi \cite{Calabi} (also see \cite{Peter}).
\begin{defn}
\label{support sense}
Let $u \in C(\Omega)$. We say $D^2 w \geq \beta g, \beta \in \R$ in support sense at $x_0$ if for every $\epsilon > 0$, there exists a smooth function $\varphi_{\epsilon}$ defined in a neighborhood of $x_0$ such that i) $\varphi_{\epsilon}$ touches $w$ from below at $x_0$. ii) $D^2 \varphi_{\epsilon} (x_0) \geq (\beta - \epsilon) g (x_0) $. Similarly, one define $D^2 w \leq \beta g$ in support sense.
\end{defn}

The following well-known property of the distance function is useful (see. p.342 \cite{Peter})

\begin{prop}
\label{Hessian bound of distance function}
Let $(\Mfd, g)$ be a (smooth) Riemannian manifold. Given any $y \in \Mfd$,  $\nabla^2 \rho_y^2$ is locally bounded above in support sense, that is, for any compact set $Z$, there exists a constant $\mathcal{L}$(depending on $\diam{Z}$ and the sectional curvature lower bound over $Z$), such that 
\[
D^2 \rho_{y}^2  (x)  \leq   \mathcal{L}\;  g , \quad \forall x\in Z.
\]
in support sense. 
\end{prop}

\begin{rem}
$C^2$-functions are locally bounded above in support sense.
\end{rem}

The following lemma contains a key feature of the contact sets.
\begin{lem}
\label{cut-focal free}
Let $u \in C(\overline{\Omega})$. Let $E \subset \Omega$ be closed. Let $a \geq 0$. Suppose $A (a, E/\Omega, u) \subset \Omega$ and $u$ is locally bounded above in support sense in $\Omega$ , then the following statement holds. 

if $y \in E$ and paraboloids $P_{a, y}$ touches $u$ at $x \in A (a, E/\Omega , u)$, then $x$ and $y$ are neither cut-points nor focal points for each other and hence $P_{a,y}$ is smooth at $x$.  
\end{lem}

\begin{proof}
By the contact relation, the definition of bounded in support sense and Prop.\ref{Hessian bound of distance function}, there are two smooth function $\varphi^+, \varphi_-$ such that $\varphi^+$ touches $P_{a,y}$ from above at $x$ and $\varphi_-$ touches it from below at $x$. It follows immediately that $P_{a,y}$ is differentiable at $x$. By the standard Riemannian geometry, $x$ and $y$ are not cut-points of each other. 

To show $x,y$ are not focal points to each other, consider the limit of the second order increment quotient
\[
\Delta^2 P_{a,y} (x):= \limsup_{\abs{W}\rightarrow 0} \frac{P_{a,y} (\exp_{x} W) + P_{a,y} (\exp_{x}-W) - 2 P_{a,y} (x)}{\abs{W}}.
\]
The existence of $\varphi^+, \varphi_{-}$ shows that
\[
\abs{\Delta^2 P_{a,y} (x) }< \infty
\]
By Prop.2.5 of \cite{MDS}, this eliminates the possibility that $x, y$ are focal points to each other.
\end{proof}

\medskip
Next lemma relates contact sets, the sub-level sets of $u$ and the domain. Such a statement has indeed been used in \cite{Cab}.

\begin{lem}
\label{contact location} 
Let $u \in C(\overline{B}_{r} (x_0 ))$. Let $u (y_0) = l$ for some $y_0 \in \overline{B}_{r/2} (x_0)$ and $u \geq t $ in $B_{5r/6} (x_0)$. 

Suppose $l < t$. then for any $a \geq 0$
\[
A (a, \overline{B}_{r/6} (y_0) / B_{r} (x_0), u) \subset B_{5r/6} (x_0) \cap \{ u\leq l + \frac{ar^2}{36} \}
\]
\end{lem}

\begin{proof}
Let $P_{a, y_1}$ be a polynomial touching $u$ at $x_1$. By the contact relation, we have 
\begin{equation}
\label{touch 1}
u(y_0) \geq P_{a, y_1} (y_0) = -\frac{a}{2} \rho^2 (y_0, y_1) + \frac{a}{2} \rho^2 (x_1, y_1) + u (x_1).
\end{equation}
Thus, we immediately have
\[
u (x_1) \leq P_{a,y_1} (y_0) + \frac{a}{2} \rho^2 (y_0, y_1) \leq l + \frac{ar^2}{36}.
\]
So it suffices to show $x_1 \in B_{5r/6} (x_0)$.

Suppose on the contrary that $x_1 \in B_{r}(x_0) \setminus B_{5r/6} (x_0)$, then, by $y_0 \in B_{r/2} (x_0)$ and $y_1 \in B_{r/6} (y_0)$, we have
\[
\rho (x_1, y_1) \geq \frac{r}{6},
\]
Henceforth
\begin{equation}
\label{touch 2}
\rho^2 (x_1, y_1) - \rho^2 (y_0, y_1) \geq 0
\end{equation}

However, by (Eq.\ref{touch 1})
\begin{equation}
\label{touch 3}
 \frac{a}{2}\bigl( \rho^2 (x_1, y_1) -  \rho^2 (y_0, y_1) \bigr)\leq u (y_0) - u (x_1) \leq l- t < 0
\end{equation}
Since $a \geq 0$, (Eq.\ref{touch 2}) contradicts to (Eq.\ref{touch 3}).
\end{proof}

\begin{rem}
Though sufficing for this paper, this is not the most precise relation between contact sets and sub-level sets. However, the proof of above lemma has suggested how one might control the relative position of contact sets, sub-level sets and domain.
\end{rem}

\begin{rem}
On space with special feature in metric geometry, such as Euclidean space which has parallelogram law, very precise relation can be draw regarding the relative location of contact sets with different opening (see. \cite{Savin}).
\end{rem}

\bigskip

\section{Jacobi Equation and Jacobi Determinant}
In this section, we quote some important results regarding the Jacobi equation and their geometrical implication from \cite{Villani}. They are of fundamental importance in our development. In particular, we shall see how contact sets match with the Jacobi determinant (Prop.\ref{geometric Jacobi}). The content in this section follows closely to the Chapter 14 (p.365-372, p.379-383) in \cite{Villani} and its third Appendix (p.412-418 ).

\begin{defn}
Let $R (t)$ be a $t$-dependent symmetric matrix. The Jacobi equation associate to $R(t)$ is the following ODE
\begin{equation}
\label{Jacobi eq}
\ddot{J} (t) + R (t) \cdot J (t) = 0.
\end{equation}
Solutions to Eq.\ref{Jacobi eq} are called Jacobi matrices.
\end{defn}
In the sequel, we shall always assume the time interval to be $[0,1]$.

The following propositions contain the main properties of the Jacobi equation that supports our development (proof, see. p.429-432, \cite{Villani}).

\begin{prop}
\label{positivity of S(t)}
Let $J^{1}_0$and $J^0_1$ be Jacobi matrices defined by the initial conditions
\[
J^{1}_0 = \dot{J^0_1} = I, \quad \dot{J}^1_0 = J^0_1 = 0. 
\]
Assume $J^0_1$ is invertible for all $t\in (0, 1]$. Then,
\[
S (t) := [ J^0_1 (t)]^{-1} J^1_0 (t)
\]
is symmetric for all $t\in (0,1]$ and decreases monotonically.
\end{prop}

\begin{rem}
The original statement in the book \cite{Villani} also states that $S(t)$ is positive for all $t\in [0, 1)$. We have confirmed with the author that the statement $S(t)$ is positive is merely a typo. In particular, the material in the third Appendix of Ch.14 in \cite{Villani} does not rely on positivity of $S(t)$.
\end{rem}

\begin{prop}
\label{good initial condition}
Let $S(t)$ be the matrix defined in Prop.\ref{positivity of S(t)}. Let $J (t)$ be a Jacobi matrix satisfies the initial conditions
\[
J (0) = I, \quad  \dot{J} (0) \text{ is  symmetric}
\]
Then the following properties are equivalent

i) $\dot{J} (0) + S(1) \geq 0$;

ii) $\det J (t) > 0 $ for all $ t \in [0,1)$.
\end{prop}

Now, we related the above pure ODE results to some geometry of Jacobi fields on Riemannian manifold. The following discussion follows closely to \cite{Villani}.

Let $(\Mfd, g)$ be a Riemannian manifold, given a geodesic $\gamma (t)$, one may parallel transport an orthonormal frame $e(0)$ at $T_{\gamma (0)} \Mfd$ along $\gamma (t)$ to obtain frame $e (t) $ at $T_{\gamma (t)} \Mfd$. Then, a family of sections $H (t) \in \mathrm{Sym} T_{\gamma (t) } \Mfd  $ can be canonically identify to a family of symmetric matrices parametrized by $t$. In the rest of this paper, we shall always use this identification whenever necessary. Eigenvalues of $H (t)$ are independent of choice of the frame $e(t)$.

Consider the flow $F[u] (t, \cdot)$:
\begin{defn}
\label{map F[u]}
Let $u \in C^2(\Omega)$. Define
\[
F[ u] (t , \cdot ) : \Omega \rightarrow \Mfd ,  x \mapsto \exp_{x} ( t\nabla u (x)).
\]
For our convenience, we shall denote $F[u] = F[u](1, \cdot)$; and the notation $F_{t} [u]$ and $F [u] (t, \cdot)$ are used interchangeably. We also denote the Jacobi transformation $d F[u]$ by $J[u]$.
\end{defn}

The next proposition contains some geometric implication of the previous two propositions (see the discussion on p.413-414 in \cite{Villani}).

\begin{prop}
\label{geometric Jacobi}
Let $x, y \in \Mfd$. Suppose $x, y$ are neither cut-points nor focal points to each other. Then following statements holds:

i)  $J[u] (t, x)$ is a smooth (w.r.t time $t$) Jacobi matrix associated to 
\[
R^i_j (t,x )= \mathrm{Riem} (\dot{\gamma} (t,x), e_i (t), \dot{\gamma} (t,x), e_j (t) ).
\]
with initial conditions
\[
J (0, x) = I, \quad \dot{J } (0,x) = \nabla^2 u (x).
\]
along the cure
\[
\gamma [u] (t ,x ):= \exp_{x} (t \nabla u (x))
\]

\smallskip

ii) if
\[
\nabla^2 u (x) + \nabla^2 (\frac{1}{2}\rho_{y}^2 )  (x) \geq 0
\]
Then, $\det J[u] (t,x ) \geq 0 $ for all $t \in [0,1]$.

\smallskip

\end{prop}

\begin{proof}
The proof is contained in Ch.14 of \cite{Villani}. See the discussion on p.365-- p.367 for i) and p.412-- p418 on ii) 
\end{proof}

\begin{rem}
Note, this is a proposition regarding a Riemannian Manifold. No reference measure appear.
\end{rem}

\begin{rem}
Jacobi fields has been used in \cite{Cab} to give an explicit expression of $d F[u]$. Prop.\ref{geometric Jacobi} can be deduced from that calculation. However, we shall not need the expression of $d F[u]$.
\end{rem}

\medskip

Next, we shall estimate the Jacobi determinant of $F[u]$ with respect to reference measure $\nu$. Materials and their proofs can be found on Ch.14 of \cite{Villani}.

The Jacobi equation (\ref{Jacobi eq}) immediately suggests that behavior of $J[u](t)$ is controlled by curvatures and the Hessian of $u$. However, we are only interested in estimating the Jacobi determinant, which can indeed be controlled by Ricci. 

First, we make few definitions and notations.
\begin{defn}
\label{Jacobi determinant}
Let $(\Mfd, g, \nu)$ be a complete Riemannian Metric-Measure space and $\Omega \subset \M$ be a domain. Let $u \in C^2(\Omega)$. We denote
\[
\J[u] (t ; x ) := \det (J[u] (t, x)), \quad  x \in \Omega.
\]
and define
\[
\J_{\nu}[u] (t , x) :=  \lim_{r \rightarrow 0} \frac{ \nu [ F_t[u]  (B_{r} (x) ) ]   }{ \nu [B_{r} (x)]     }=\frac{e^{-V (F_t[u] (x) ) } }{e^{-V(x)}} \J[u] (t, x), \quad  x \in \Omega
\]
Also, we denote
\[
\mathcal{D}_{N} [u] (t,x)  = \begin{cases}
 (\J_{\nu} [u] (t,x) )^{1/N}    & n \leq N < \infty \\ 
\log \J_{\nu}[u] (t,x) &  N = \infty 
\end{cases}  .
\]
\end{defn}

The following propositions quoted from Ch.14 of \cite{Villani} are the keys of estimating $\J_{\nu} [u]$,

\begin{prop}
\label{det estimate1}
Let $u \in C^2(\Omega)$. Let $x \in   \Omega$ and $\gamma (t, x)$ is a geodesic starting at $x$ with $\dot{\gamma} (t ,x ) = \nabla u (x)$. Suppose $J[u] (t, x)$ is invertible for all $t \in  [0,1)$, then
\[
\ddot{\D}_{n}[u] (t,x )  \leq - \frac{\Ric (\dot{\gamma} (t,x))}{n} \D_n [u] (t,x), \forall t\in (0,1)
\]
\end{prop}

\begin{proof}
See page 368-370 of \cite{Villani}.
\end{proof}

In the presence of reference measure, the techniques on page 380 \cite{Villani} extends the above proposition as follows 

\begin{cor}
\label{det estimate2}
Let $u \in C^2(\Omega)$. Let $x \in \Omega$ and $\gamma (t, x)$ is a geodesic starting at $x$ with $\dot{\gamma} (t ,x ) = \nabla u (x)$. Suppose $J[u] (t, x)$ is invertible for all $t \in  [0,1)$, then, for any $N \in [n ,\infty]$, 
\[
\ddot{\D}_N [u] (t, x) \leq \begin{cases}
- \frac{1}{N}\Ric_{N, \nu} (\dot{\gamma} (t,x))  \D_N  [u] (t,x) &  n\leq N < \infty \\
 -\Ric_{\infty, \nu} (\dot{\gamma } (t,x )) &  N = \infty
\end{cases}, \forall t \in (0,1)
\]
\end{cor} 
\begin{proof}
See p.379-383 in \cite{Villani}.
\end{proof}

\begin{rem}
Our definition of $\D_{\infty}$ differs a sign from the function $l(t)$ given in \cite{Villani}.
\end{rem}

\bigskip

\section{Measure Estimate}
In this section, we give the proof of Thm.\ref{Measure Estimate}. Indeed, all the technical work has been done in the previous two sections.

\begin{proof}[Proof of Thm.\ref{Measure Estimate}]
For convenience, write $A = A(a, E/B_{r}, u )$ in this proof. First, we show that the map $F[u](x) = F[ u/a]  (1,x)$ is a subjective map from $A$ onto $E$. 

Fix a point $y \in E$, by definition, there exists a paraboloid $P_{a, y}$ touches $u$ at some $x \in A $. By Lemma \ref{cut-focal free}, we have $P_{a, y}$ is smooth at $x$ and the contact condition implies 
\begin{equation}
\label{gradient identity}
\nabla u (x) = -a \rho_{y} (x ) \nabla \rho_{y} (x)
\end{equation}
Hence, 
\[
F[u/a] (1, x) = \exp_{x} [  -\rho_{y} \nabla \rho_y ( x)   ]  = y.
\]
This proves the subjectivity.

Next, as indicated by Lem.\ref{cut-focal free} $x$ and $y$ are neither cut-points nor conjugate points of each other, Prop.\ref{geometric Jacobi} along with the contact relation
\[
\nabla^2 \frac{u}{a} (x) \geq -\nabla^2 \frac{1}{2}\rho_y^2 (x)
 \]
imply that $J[u] (t,x)$ is invertible for all $t\in (0,1)$ and
\begin{equation}
\label{det J nonnegative}
\J_{\nu} (t, x) = \frac{e^{-V(F[u] (t,x))}}{e^{-V (x)}}\det J[u] (t,x ) \geq 0 \quad \forall t \in [0,1].
\end{equation}
Therefore, $\D_{N}$ (recall Defn.\ref{Jacobi determinant}) satisfies the differential inequality given in Cor.\ref{det estimate2}.

Denote $\gamma (t,x) = \exp_{x} (t \nabla u (x)/a )  $, by Eq.(\ref{gradient identity}) and the fact $\gamma$ is a geodesic, we have
\[
\abs{\dot{\gamma} (t,x )}^2 = \abs{\dot{\gamma} (0,x) }^2 = \rho_{y}^2 (x) \leq 4r^2, \quad \forall t \in [0,1],
\]
where the last inequality follows from the fact that $A  \subset B_{r}$.

Therefore, along with the Ricci lower lower bound condition, the differential inequality in Cor.\ref{det estimate2} reduce to 
\begin{equation}
\label{Ricci ode}
\ddot{\D}_N[u/a]  (t ,x) \leq \begin{cases}
   4 (K/N) r^2\; \D_{N}[\frac{u}{a}] (t ,x )   &   n\leq N  < \infty\\ 
  4K r^2 & N = \infty
\end{cases}. 
\end{equation}
Now apply a standard ODE comparison argument with the initial condition that 
\[
\D_N[u/a] (0,x ) = 
\begin{cases}
1  & n \leq N < \infty \\
 0 &  N = \infty
\end{cases}
\quad 
\dot{\D}_{N}[u/a] (0, x) =
\begin{cases}
 \frac{1}{Na} \Delta_{\nu}u (x)   & n \leq N < \infty \\
  \frac{1}{a} \Delta_{\nu} u(x) &  N = \infty
\end{cases},
\]
we obtain
\begin{equation}
\label{final det estimate}
\D_{N}[u/a] (1, x) \leq \D_{K,N, R} [u/a] (x), \quad \forall x \in A,  N \in [0, \infty]
\end{equation}

\smallskip

Finally, we shall apply the Area formula. Since $u$ is $C^2$ in $\Omega'$ containing $A$,  $F[u/a] (1, \cdot)$ is differentiable in $\Omega'$. Thus, by the Area formula, 
\[
\nu [E] \leq \int_{ A } \J_{\nu} (x)\; \nu(dx) \leq 
\int_{ A }  \biggl( \D_{N}[u/a] (1, x)  \biggr)^{N} \; \nu (dx) \quad n \leq N < \infty
\]
and
\[
\nu [E]  \leq \int_{ A }  \J_{\nu} (x) \;\nu(dx) \leq \int_{ A }  \exp \biggl(  \D_{\infty}[ u/a] (1, x)  \biggr) \; \nu (dx) , \quad N = \infty
\]
Here, we have used (\ref{det J nonnegative}) and the fact that $\J_{\nu}  \geq 0 $ (Eq.\ref{det J nonnegative}). The desired formula follows from (Eq.\ref{final det estimate}).
\end{proof}

We provide some remarks on this estimate. 
\begin{rem}
Note in the proof, the contact set is exactly where one can estimate Jacobi determinant of $F[u]$--the place where $\det (dF[u] )$ is nonnegative! This has already been observed by Cabr\'e. In \cite{Cab}, one focus on working with sub-level sets, contact sets are used only as an intermediate step. 
\end{rem}

\begin{rem}
There are only two inequalities used in above proof. One is during the application of area formula. It reaches equality if and only if $\F[u]$ is one-to-one. Another one is the estimate by the ODE. This differential inequality (Eq.\ref{Ricci ode}) is, indeed, equivalent to the Ricci lower bounded ( if it is satisfied by all suitable test functions, see details on p.400, Prop.14.8 in \cite{Villani}).
\end{rem}

\begin{rem}
The formula given by Thm.\ref{Measure Estimate} is, in certain sense, some dual formula to the Sobolev inequalities. In particular, the case $N = \infty$ could be viewed as a dual formula for the log-Sobolev inequality.

One way to recognize the duality is to consider the key ingredients in the proof of Thm.\ref{Measure Estimate} and the proof of Sobolev inequalities. It is known that Sobolev inequalities can be derived as consequence of Co-area formula. While Thm.\ref{Measure Estimate} is proved based on the Area formula.

Alternatively, one knows that Harnack inequalities are equivalent to Sobolev inequality on the manifolds satisfying doubling property (see,\cite{SC} and \cite{Kim}
), and we shall see in \S \ref{key growth lemma} -- \S \ref{soln}, Thm.\ref{Measure Estimate} along with the measure doubling property implies the Harnack inequality. Thus, there must be some relations between Thm.\ref{Measure Estimate} and Sobolev inequalities.
\end{rem}

\bigskip

\section{Ricci Comparison and A Barrier}

In this section, we recall the Ricci comparison theorem for the modified Laplacian and use it to construct a barrier function which shall be used latter (proof of Lem.\ref{decay}).

Recall the following result in \cite{Qian} (or \cite{LiXD}, \cite{WW}). 
\begin{prop}
\label{Ricci comparison}
Suppose $\Ric_{N, \nu} |_{B_{R}} \geq - Kg, K \geq 0, N \in [0, \infty)$, then for any two points $x, y \in B_{R}$. 
\[
\Delta_{\nu} \rho_{y} (x) \leq  (N-1) \frac{\tc (\omega_{K, N-1} \rho_y (x))}{ \rho_y (x) }
\]
in the support sense everywhere (recall notation \ref{tc, st} from \S\ref{notations}).
\end{prop}

\begin{rem}
Note, it is easy to check
\[
1+ (N-1) \tc(\omega_{K, N-1} \rho) \leq N \tc (\omega_{K, N} \rho)   
\]
Hence
\[
\Delta_{\nu} \frac{\rho^2_y}{2} (x) \leq N \tc (\omega_{K,N} \rho_y (x))
\]
\end{rem}

\medskip

The rest of this section is devoted to construct a barrier function. A similar construction has been made in \cite{Cab} (also adopted in \cite{Kim}). However, as working with potentially negative curvature,  one needs some more detailed information regarding such a barrier function to insure the constants depends on curvatures in a propert way (in particular, this is needed for considering elliptic fully-nonlinear PDEs). These information are obtained in the following lemmas. Their proofs are technical but completely routine.

Recall the definition of constant $\alpha$ from \S \ref{notations}. Note $\alpha \geq 2$.

\begin{lem}
\label{barrier 1}
There exists a function $h : [0, \infty) \rightarrow \R$ such that

i) $h \in C^2 [0,\infty) $ and $h ' (0) = 0$.

ii) $\inf_{[0,\infty)}h  \geq -\alpha^2 (18)^{\alpha}$

iii) The derivatives of $h$ satisfies the following estimates:

For $t > 1/18$, 
\[
h'' (t) - \frac{h'(t)}{t} = -\alpha (\alpha +2) t^{- (\alpha+2)}< 0, \quad, \frac{h'(t)}{t} = \alpha t^{ -(\alpha+2)}> 0
\]

For $0 \leq t \leq 1/18$, 
\[
\abs{ h'' (t) - \frac{h'(t)}{t} }  \leq 972\;  \alpha^2 (18)^{\alpha} \quad, 0 < \frac{h'(t)}{t} \leq 972 \;\alpha^2 (18)^{\alpha }, 
\]
\end{lem} 

\begin{proof}
Let $\beta_i, i= 0,1,2$ be constants to be determined.  Consider the function
\[
h (t) : =\begin{cases}
\beta_0  + \beta_1 t^2 + \beta_2 t^3   &   t \leq \frac{1}{18}  \\
 (1/9)^{-\alpha} - t^{-\alpha}  &   t > \frac{1}{18}
\end{cases}, \quad 
\]

By choosing 
\begin{equation}
\label{taylor constants}
\beta_0 = -\frac{1}{6} \alpha(5 + \alpha) (18)^{\alpha}, \; \beta_1 = \frac{(18)^2}{2}\alpha (3+\alpha) (18)^{\alpha}, \; \beta_2= -\frac{(18)^3}{3} \alpha (2 + \alpha) (18)^{\alpha}
\end{equation}
we match up the values and the first two derivatives of $h$ at  $t = 1/9$. It follows then $h$ satisfies i).

\smallskip

Next, we estimate the derivatives of $h$. The case for $t > 1/9$ is clear. Consider $t \in [0, 1/9]$
 
Since $\alpha \geq 0$,
\[
\beta_0, \;\beta_1 \geq 0 \quad \beta_2 < 0.
\] 
and,
\[
  - \frac{\beta_1}{\beta_2}  = \frac{1}{12} \frac{ (3 +\alpha)}{(2 + \alpha)} > \frac{1}{12}.
\]

Therefore, we see, 
\[
\frac{h'(t)}{t} =2 \beta_1 + 3 \beta_2 t
 \]
is strictly positive for $ t \in  [0, 1/18]$ and monotone decreasing; and
\[
h'' (t) -\frac{h'(t)}{t} = 3 \beta_2 t 
\]
monotone decreasing and negative. The desired estimates follows from the expression of $\beta_i$'s and the fact $\alpha \geq 2$.

ii) follows from $h'(t) \geq 0$ in $0 \leq t \leq 1/18$ and the fact $\alpha \geq  2$.
\end{proof}

\smallskip

Recall the definition of $\omega_{K,N}$ and $\alpha$ from \S \ref{notations}. In the statement and the proof of next lemma, we shall denote $\omega_{K,N}$ by $\omega$.

\begin{lem}
\label{barrier 2}
Let $(\Mfd, g, \nu)$ be a complete Riemannian Metric-Measure space. Given geodesic ball $B_{r} (x_0)$ 
with $r \leq R$
. Let $K \geq 0, 1 < N <\infty$. Suppose
\[
\Ric_{N, \nu} |_{B_{r} (x_0)} \geq - K g
\]

Then there exists a function $\psi$ such that:

i) $\psi $ is continuous in $B_{r} (x_0)$ and lies in $C^2 ( B_{r} (x_0) \setminus Cut (x_0) ) $. 

ii) $\inf_{B_{r} (x_0))} \psi \geq -\alpha^2 (18)^{\alpha }$ and 
\[
\psi \geq (18)^{\alpha} - (4/3)^{\alpha} \text{ in } B_{r}(x_0)\setminus B_{3r/4} (x_0), \quad \psi = (18)^{\alpha} - 2^{\alpha} \text{ on } \partial B_{r/2} (x_0)
\]

iii) $\psi $ is locally bounded above in support sense in $B_{r} (x_0)$

iv) In $\overline{B}_{r/18} (x_0) \setminus Cut (x_0)$,
\[
\frac{ r^2 \Delta_{\nu} [\psi]}{N} +  \tc(\omega r)  \leq 972 \alpha^3  4^{\alpha} 
\]

v) In $B_{r} (x_0) \setminus ( \overline{B}_{r/18} (x_0) \cup Cut (x_0) ) $
\[
\frac{ r^2 \Delta_{\nu} [\psi] }{N}  +   \tc(\omega r)\leq 0
\]
\end{lem}

\begin{proof}
Let $h (t)$ be the function given in Lem.\ref{barrier 1} with $\alpha $ given in Eq.\ref{C alpha}. Denote $\rho_{x_0} $ by $\rho$ for convenience. Define
\[
\psi := h (\rho/ r).
\]
We shall show $\psi$ satisfies all desired properties.

Fix a poin $x \in B_{r}(x_0)  \setminus  Cut (x_0)  $, all the expressions below are evaluated at this $x$. By direct calculation, we have
\begin{equation}
\label{Hessian of barrier}
r^2\nabla^2 \psi = \biggl(  h'' -\frac{h'}{t}  \biggr) \nabla \rho \otimes \nabla \rho +  \frac{h'}{t} \nabla^2 \frac{\rho^2}{2}, \quad t  =\rho/r
\end{equation}
iii) follows from above equation and the standard argument of Calabi (p.282 \cite{Peter}).

Follows immediately from Eq.\ref{Hessian of barrier} and $\tr (\nabla \rho \otimes \nabla \rho) = \abs{\nabla \rho}^2= 1$, we have
\begin{equation}
\label{Lap on barrier}
\begin{split}
r^2 \Delta_{\nu} [\psi]  &  =  \biggl(   h'' - \frac{h'}{t} \biggr)  + \frac{h'}{t} \Delta_{\nu} \frac{\rho^2}{2}
\end{split}
\end{equation}

\smallskip

If $x \in B_{r}(x_0) \setminus ( B_{r/9} (x_0 ) \cup Cut (x_0 ))$, then by iii) of Lem.\ref{barrier 1}, the inequality (\ref{Lap on barrier}), and the Ricci comparison \ref{Ricci comparison}, 
\[
r^2 \Delta_{\nu} [\psi]   \leq \alpha \biggl( \frac{\rho}{r} \biggr)^{- (\alpha +2)} \biggl( -\alpha  - 2+ N\tc(\omega r)  \biggl) 
\]
Since $\alpha =N\tc(\omega R) \geq N \tc (\omega r) > 1$, 
\[
r^2  \Delta_{\nu} [\psi]   \leq \alpha \biggl( -\alpha  - 2+ N\tc(\omega r)  \biggl) 
\]

Thus
\[
\begin{split}
\frac{r^2  \Delta_{ \nu}  [\psi]}{N} +  \tc(\omega r)  \leq \frac{\alpha}{N} ( -\alpha - 2  + N \tc(\omega r) + \frac{N\tc(\omega r)}{\alpha})
\end{split}
\]

v) follows from the very choice that $\alpha =N\tc(\omega R) \geq  N \tc (\omega r ) $ (Eq.\ref{C alpha}), 

\smallskip

If $x \in B_{r/18} (x_0) \setminus Cut (x_0)$, then by the same calculation this time with the other part in iii) of Lem.\ref{barrier 1}, we arrive (Note $N > 1$)
\[\begin{split}
\frac{r^2  \Delta_{\nu}  [\psi]}{N} & +  \tc(\omega r) \leq  \frac{972}{N}\alpha^2 9^{\alpha} + (972 \alpha^2 (18)^{\alpha} + 1 ) \tc(\omega r)
\end{split}\]
iv) follows immediately from the choice of $\alpha $.
\end{proof}

\bigskip

\section{Measure Doubling and Monotonicity of $\I_{K,N}$}
\label{integral}
In this section, we summarize some results regarding the measure doubling property, the integral $\I_{K,N}$ and some basic $L^p$ theory that will be used in the proof of Harnack inequalities (Thm.\ref{Harnack sup finite}--Thm.\ref{Harnack soln finite}).

The following proposition estimates the doubling constant on a Riemannian Metric-Measure space in terms of Ricci lower bound (Cor.18.11, \cite{Villani}). 

\begin{prop}
\label{doubling finite}
Let $(\Mfd, g, \nu)$ be a Riemannian metric-measure space, satisfying the curvature condition $\Ric_{N, \nu}\geq -K$ for some $K \geq 0$ and $1<N< \infty$. Denote $\db_{\Omega}$ the doubling constant in the domain $\Omega$. Then $\nu$ is doubling with a constant $\db$:
\begin{itemize}
\item $\db_{\Mfd} \leq 2^N$ if $K=0$;
\item $\db_{B_{R} } \leq 2^N \Big[\cosh\Big(2\sqrt{\frac{K}{N-1}}R\Big)\Big]^{N-1} \leq \db_{K,N,R}$ for any $B_{R}$ if $K>0$
\end{itemize}
(recall Eq.\ref{C omega, db} in \S \ref{notations}).
\end{prop}

The doubling property allows the following simple estimate. Recall $\eta=\eta_{K,N,R}$ and $\db = \db_{K,N,R}$ (Eq.\ref{C omega, db}. Eq.\ref{C eta}) from \S \ref{notations}
\begin{equation}
\label{doubling estimate}
\frac{\nu [B_{r_1}(x)]}{\nu [B_{r_2}(x)]} \leq  \db\biggl( \frac{r_1}{r_2} \biggr)^{N\eta}
\end{equation}
provide $B_{r_2}(x) \subset B_{r_1}(x) \subset B_{R}$ on a $(\Mfd, g, \nu)$ with $\Ric_{N,\nu} |_{B_{R}} \geq - K g$.

\smallskip 

Recall the definition of $\I_{K,N}$ (Eq.\ref{KN integral}) from \S \ref{notations}. The integral $\I_{K,N} $ has good monotonicity and fits the scaling well.
\begin{lem}
\label{monotone integral}
For any $1 < N < \infty$, we have

i) $\I_{N} (f ; B_{r}, t) \leq \I_{N } (f; B_r, s) $ whenever $t < s$.

ii) If $B_{r_1}(x) \subset B_{r_2} (x) \subset B_{R} $, then 
\[
\I_{N} (f ; B_{r_1}(x), \eta_{K,N,R}) \leq \I_{N} (f ; B_{r_2}(x), \eta_{K,N,R}).
\]
(recall the definition of $\eta_{K,N,R}$ (Eq.\ref{C eta}) from \S \ref{notations}).
\end{lem} 

\begin{proof}
i) follows from the standard $L^p$ theory; ii) follows from direct calculation and Eq.\ref{doubling estimate}.
\end{proof}

\medskip

Given a function $f$ on a domain $\Omega$ with finite-measure, denote
\begin{equation}
\label{distribution}
\tilde{\lambda}_{\Omega} (t) := \frac{\nu [\{ f \leq t \} \cap \Omega]}{\nu [\Omega]}.\end{equation}
When no confusion arise, we shall omit the subscript. 

We shall need the following well-known statement in $L^p$-theory (see \cite{CC} for instance). 
\begin{lem}
\label{Lp}
Let $C> 1$. Then, for any $0 < p < \infty$, 
\[
 \fint_{\Omega} f^{p} < \infty  \Leftrightarrow  S:= \sum_{k = 0}^{\infty} C^{pk} \tilde{\lambda} (C^{k}) < \infty
\]
and 
\[
(1- \frac{1}{C^{p}}) S+ \frac{1}{C^{p}}\tilde{\lambda} (1) \leq\fint_{\Omega} f^{p} \leq 1+  (C^{p} - 1)S
\]
\end{lem}

\bigskip

\section{Proof of Harnack Inequalities I}
\label{key growth lemma}
In this section, we establish the key lemma in proving Harnack inequalities (Thm.\ref{Harnack sup finite}--Thm.\ref{Harnack soln finite}). It describes the local growth of the unknown $u$. The similar lemma is used in \cite{Cab}. Our proof is essentially same as that in \cite{Cab} (also in \cite{Kim}). However, by making use some fine properties of the contact sets, we avoid the approximation procedures needed in \cite{Cab} and \cite{Kim}.

Recall from \S \ref{notations}, the constant $M, \mu$ (Eq.\ref{C mu}) and integral $\I_{K,N}$ (Eq.\ref{KN integral}).

\begin{lem}
\label{local growth}
Let $(\Mfd, g,\nu)$ be a complete metric-measure space. Let $u \in C (\overline{B}_{2R} ) \cap C^2 (B_{2R} )$ and $f \in C (B_{2R} )$. Let $K \geq 0 , N < \infty$. 

Suppose  
\[
\Ric_{N, \nu} |_{B_{2R} }\geq - K g, \quad \I_{K, N} (f, B_{2R}, 1) \leq \delta_0
\]

Then, for any given ball $B_{2r}(x_0) \subset B_{2R}$, 
\[
u \geq 0 \text{ in } B_{r} (x_0), \quad \; \inf_{B_{r/2}} u \leq 1 , \quad \Delta_{\nu} [u] \leq f  \text{ in } B_{r} (x_0),
\]
implies 
\begin{equation}
 \frac{\nu[\{u\leq M \}\bigcap B_{r/18}(x_0) ]}{\nu[B_{r}(x_0)]}\geq \mu
\end{equation}
\end{lem}

\begin{rem}
Recall the \textbf{Notice} in \S \ref{notations}, $M, \mu , \delta_0$ are constants defined w.r.t to the large ball $B_{2R}$. They are independent of the choice of the small ball $B_{r} (x)$.
\end{rem}

%
%
%
%

\begin{proof}
Recall the definition of the constant $\alpha$ (\S \ref{notations} Eq.(\ref{C alpha})). Let $\psi$ be the function constructed in Lem.\ref{barrier 2} with respect to $B_{r} (x_0)$. Consider $w = u + \psi$ on $\subset B_{r} (x_0)$.

By the construction of $\psi$ and the hypothesis of $u$, there exists $y_0 \in B_{r/2} (x_0)$ such that
\[
w(y_0) = \inf_{B_{r/2}(x_0) }w \leq 1 + (18)^{\alpha} - 2^{\alpha}, 
\]
and $w$ satisfies
\[
\inf_{B_r (x_0) \setminus B_{3r/4}(x_0)}w \geq (18)^{\alpha} -(\frac{4}{3} )^{\alpha}.
\]

This two condition along with Lem.\ref{contact location} and $\alpha \geq 2$ implies
\begin{equation}
\label{location Aw}
 A\Big(  \frac{1}{r^2} , \overline{B}_{r/6}(y_0) / B_{r} (x_0) , w\Big)\subset\subset
B_{r} (x_0)\bigcap \{w \leq  1+ (18)^{\alpha} - 2^{\alpha}+ \frac{1}{36} \}.
\end{equation}
For convenience  in the rest of the proof, we denote 
\[  
A:=  A\Big(  \frac{1}{r^2} , \overline{B}_{r/6}(y_0) / B_{r} (x_0) , w\Big),
\] 

Recall ii) of Lem.\ref{barrier 2}, we obtain
\[
w = u + \psi \leq \frac{37}{36} +(18)^{\alpha} -2^{\alpha} \Rightarrow  u \leq 2 \alpha^2 (18)^{\alpha}
\]
This along with (Eq.\ref{location Aw}) and the definition of $M$ (\S \ref{notations},Eq.\ref{C M}) implies
\begin{equation}
\label{A in sublevel}
\nu [A \cap B_{r/18} (x_0)] \leq \nu[ \{ u < M\} \cap B_{r/18} (x_0) ].
\end{equation}
Hence, it suffices to estimate $\nu [A \cap B_{r/18} (x_0)] $ from below.

Recall the definition of $\db_{K,N,r}, \omega_{K,N}$  (Eq.\ref{C omega, db}) and the function $\tc(t), \st(t)$ (Eq.\ref{tc, st}) from \S \ref{notations}. In the rest of this proof, we shall denote $\omega_{K,N}$ by $\omega$ and $\db_{K, N, r}$ by $\db_{r}$. We will give the estimate for $\nu [A \cap B_{r/(18)} (x_0)]$ via Thm.\ref{Measure Estimate}.

\smallskip

By Lem.\ref{barrier 2}, $\psi$ is locally bounded above in support sense. Also $u$ is $C^2$, hence by Prop.\ref{cut-focal free}, we have
\[
A \cap Cut (x_0 ) = \emptyset
\]
Since $A$ is a closed (Lem.\ref{basic contact}) and $Cut (x_0) $ is also closed, there is a neighborhood $\Omega' \subset B_{r/4} (y_0)$ of $A$ such that $w \in C^2 (\Omega')$. Thus we may apply Thm.\ref{Measure Estimate} to obtain 
\begin{equation}
\label{me 1}
\nu [B_{r/6} (y_0) ] \leq \int_{A}\Big\{ \D_{K, N, r}[r^2 w] (x) \Big\}^N\;  \nu (dx) 
\end{equation}
with 
\[\begin{split}
\D_{K,N, r}[r^2 w] (x)  &  =\st (\omega r) \biggl[ \tc (\omega r)   + \frac{r^2\Delta_{\nu}w (x)}{N }   \biggr] \\
\end{split}\]

Since $\Delta_{\nu} u \leq f$, we have
\[
\label{decay ieq 1}
\tc (\omega r)   + \frac{r^2 \Delta_{\nu} w (x)}{N }  \leq \frac{r^2 f(x) }{N}+ \tc(\omega r) +\frac{r^2}{N} \Delta_{\nu}\psi (x).
\]
Lem.\ref{barrier 2} implies: for any $x \in A \cap \overline{B}_{r/18} (x_0)  $
\begin{equation}
\label{D estimate in}
\D_{K,N, r}[r^2w] (x) \leq  972 \alpha^{3}4^{\alpha}  \st (\omega r)     + \st (\omega r) \biggl( \frac{r^2 f^+ (x)}{N} \biggr);
\end{equation}
and for any $x \in A \cap (B_{r} (x_0) \setminus \overline{B}_{r/18} (x_0)) $, 
\begin{equation}
\D_{K,N, r}[r^2w] (x)  \leq \st(\omega r) \biggl( \frac{r^2 f^+ (x)}{N} \biggr).
\end{equation}

These estimates along with the simple relation 
\[
(t + s)^{N} = (t+s)^{+, N} \leq 2^{N-1} (t^{+,N} + s^{+, N}) ,\quad t + s > 0,
\]
 implies
\[
\begin{split}
\int_{A} \Big\{ \D_{K, N, r}[r^2w] (x) \Big\}^N\;  \nu (dx) & \leq   \frac{1}{2}\biggl[ (18)^3 \;\alpha^{2}(18)^{\alpha}\cosh(\omega r) \biggl]^{N}    \nu[A \cap B_{r/18} (x_0)] \\
& + 2^{N-1}\st^N (\omega r) \int_{B_{r} (x_0) }  \biggl( \frac{r^2 f^+ (x)}{N} \biggr)^{ N}
\end{split}
\]

Combine with (Eq.\ref{me 1}), we obtain
\begin{equation}
\label{decay final}
\begin{split}
1 & \leq \frac{1}{2} \biggl[ (18)^3\;\alpha^{2}4^{\alpha}\cosh(\omega r) \biggl]^{N}    \frac{ \nu [ A \cap B_{r/18} (x_0 ) ]  }{ \nu [B_{r/6} (y_0)] } \\
& + 2^{N-1}\st^N (\omega r) \frac{1}{\nu [B_{r/6} (y_0)]}\int_{B_{r} (x_0) }  \biggl( \frac{r^2 f^+ (x)}{N} \biggr)^{ N}
\end{split}\end{equation}
By the doubling property
\begin{equation}
\label{decay double}
\nu [B_{r} (x_0)] \leq \nu[B_{3r/2} (y_0)] \leq \db_{2r}^4 \nu [B_{r/6} (x_0)]
\end{equation}
we see, by monotonicity of $\I_{K,N}$ (Lem.\ref{monotone integral}), 
\begin{equation}
\label{decay integral}
\begin{split}
\frac{1}{\nu [B_{r/6} (y_0)]}\int_{B_{r} (x_0) }&  \biggl( \frac{r^2 f^+ (x)}{N} \biggr)^{ N} \leq \db_{2r}^4  \fint_{B_{r} (x_0) }  \biggl( \frac{r^2 f^+ (x)}{N} \biggr)^{ N} \\
&  \leq \db_{2r}^4  \biggl( \I_{K,N} (f, B_{r}, 1) \biggl)^N\leq \db_{2r}^{4} \delta^N_0
\end{split}  
\end{equation}

Therefore, by combine (Eq.\ref{decay final}) with (Eq.\ref{integral}) and recall the choice of $\delta_0$ (Eq.\ref{C delta}, in \S \ref{notations}), we arrive
\[
  \frac{1}{2} \biggl[ (18)^3 \; \alpha^{2 } (18)^{\alpha}\cosh(2\omega r) \biggl]^{N} \db_{2r}^{4}   \frac{ \nu [ A \cap B_{r/18} (x_0 ) ]  }{ \nu [B_{r} (x_0)] } \geq 1/2
\]
Recalling the definition of $\mu$ (\S \ref{notations}, Eq.\ref{C mu}) and the inequality (\ref{decay double}), we have then
\begin{equation}
\label{mu estimate}
\frac{ \nu [ A \cap B_{r/18} (x_0 ) ]  }{ \nu [B_{r} (x_0)] }  \geq  \biggl[(18)^3\; \alpha^{2} (18)^{\alpha}\cosh(\omega r) \biggl]^{-N} \db_{2r}^{-4} \geq \mu
\end{equation}

The proof is completed by jointing the above inequality with (Eq.\ref{A in sublevel}).
\end{proof}

\bigskip

\section{Proof of Harnack Inequalities II}
\label{sup}
In this section, we shall prove Thm.\ref{Harnack sup finite}. The idea here is essentially same as that in \cite{Cab}, which dates back to \cite{KS1} and \cite{KS2}. Our presentation here follows  \cite{Savin}. 

First, we recall the following version of the Vitali's covering lemma. One may refer to standard textbook in measure theory for a proof.
\begin{lem}
Let $(X, \rho, \nu)$ be a metric-measure space; Let $\mathcal{V}$ be a family of closed balls of nonzero radius in $X$ and $D$ be the collection of centers of these balls. 

Suppose
\[
\sup \{\diam{B} : B \in \mathcal{V} \} < \infty ,
\]
and $\nu$ satisfies the local measure doubling property, that is, for any compact set $Z$, $\nu$ has a doubling constant depends on $Z$. 

Then there exists a countable subcollection $\mathcal{V}'$ of $\mathcal{V}$ such that
\[
D \subset \bigcup_{B \in \mathcal{V}'} B
\]
and the collection
\[
\{ \frac{1}{4} B , B \in \mathcal{V}'\}
\]
is disjoint.
\end{lem}

\begin{rem}
The local measure doubling property is only used to insure the collection $\mathcal{V}'$ is countable. 
\end{rem}

Recall the definition of constant $M$, $\mu$ and integral $\I_{K,N} (f)$ from \S \ref{notations}. The Thm.\ref{Harnack sup finite} follows immediately from the following two lemmas

\begin{lem}
\label{correct}
Under assumption of Thm.\ref{Harnack sup finite}. Denote
\[
D_{k}:=\{x \in B_{R/2} :  u (x) \leq M^{k}\}.
\]
Suppose additionally that 
\[
\inf_{B_{R/2}} u \leq 1 \text{ and } \;  \I_{K, N} (f , B_{2R},  \eta )\leq \delta_0.
\] 
Then for any $k\geq 0$
\[
 \nu [ D_{k+1} \cap B_{r_x/4} (x) ]   \geq \mu \nu[B_{r_x} (x)] 
\]
for all $x \in B_{R/2} \setminus D_{k}$ and $r_x = \dist{x, D_{k}}$.
\end{lem}

\begin{proof}
Fix $x_0 \in B_{R/2} \setminus D_{k}$. In the rest of the proof, write
\[
r_0=  r_{x_0} = \dist{x_0, D_{k}}
\]
for convenience. Since $D_0 \neq \emptyset$, we have
\begin{equation}
\label{r1 estimate}
r_0 \leq R/2
\end{equation}

Denote $z_0$ for the center of $B_{R/4}$. Connecting $x_0$ and $z_0$ by a minimizing geodesic. Choose $y_0$ be a point on this geodesic such that
\[
\rho (y_0, x_0) = r_0/8;
\]
and consider the ball 
\[
B_{r_0/8} (y_0).
\]

By triangle inequality and the estimate of $r_0$ and the fact (minimizing geodesic)
\[
\rho (y_0, z_0) + \rho (y_0, x_0) = \rho (z_0, x_0), 
\]
we see
\[
B_{r_0/8} (y_0) \subset B_{r_0/4} (x_0) \cap B_{R/2}. 
\]
Therefore
\begin{equation}
\label{eq 1}
B_{r_0/4} (x_0) \cap D_{k+1} \supset B_{r_0/8} (y_0) \cap \{ w \leq M\}.
\end{equation}
where
\[
w =  u / M^{k}.
\]

Thus, it suffices to estimate 
\[
\nu [ B_{r_0/8} (y_0) \cap \{ w \leq M\} ] 
\]
from below. 

Consider the ball 
\[
B_{l} (y_0), \quad l  = \frac{9}{4} r_0
 \]

Firstly, as $r_0< R/2$, we have $
 l \leq \frac{9}{8} R $
and by triangle inequality, 
\[
B_{2l} (y_0) \subset B_{2R}.  
\]

Secondly, as
\[
\dist{y_0, D_k} \leq \rho (x_0, y_0) + \dist{x_0, D_k} \leq \frac{9}{8} r_0.
\]
we have, as $l/2 = 9r_0/8 $, 
\[
\overline{B}_{l/2} \cap D_{k} \neq \emptyset.
\]

Thirdly, note
\[
\frac{l}{18} = \frac{r_0}{8}  .
\]
hence
\[
B_{l/18} (y_0) \cap \{w \leq M\} = B_{r_0/8}(y_0) \cap \{w \leq M\}
\]

With these three elementary relations, we may apply the Lem.\ref{local growth} to $w$ on $B_{l} (y_0)$ to obtain
\begin{equation}
\label{eq 2}
\nu [B_{r_0/8} \cap\{w \leq M\} ] \geq \mu \nu [B_{l} (y_0)]
\end{equation}

Finally, use the triangle inequality again, we see
\begin{equation}
\label{eq 3}
B_{l} (y_0)  = B_{9r_0/4} (y_0) \supset B_{r_0} (x_0).
\end{equation}

Combine Eq.\ref{eq 1}, Eq.\ref{eq 2} and Eq.\ref{eq 3}, we complete the proof. \end{proof}


\begin{lem}
\label{decay}
Under assumption of Thm.\ref{Harnack sup finite}. Suppose additionally that 
\[
\inf_{B_{R/2}} u \leq 1 \text{ and } \;  \I_{K, N} (f ; B_{2R}; \eta )\leq \delta_0.
\] 
Then for any $k\geq 0$
\[
\tilde{\lambda}_{B_{R/2}} (M^{k}) \leq (1 - \mu)^{k}.
\]
\end{lem}

\begin{proof}
Recall 
\[
D_{k}:= \{x \in B_{R/2} :   u (x) \leq M^{k} \}. 
\]

\textit{Claim:} for any $k \geq 0$
\[
\nu [ (D_{k+1} \setminus D_{k} ) \cap B_{R/2} ] \geq \mu  \nu [B_{R/2} \setminus D_k ].
\]

Consider the cover $\mathcal{V}$ of the set $B_{R/2} \setminus D_{k}$ defined by 
\[
\mathcal{V} := \{ \overline{B}_{r_x} (x) : x \in B_{R/2} \setminus D_{k},  r_{x}:= d (x, D_{k}) \}.
\]
By Lem.\ref{correct}, we have 
\[
  \nu [ D_{k+1} \cap B_{r_x/4} (x) ]   \geq \mu \nu[B_{r_x} (x)] 
\]

By Vitali's covering lemma, we may take a sequence of ball $B_{r_i} \in \mathcal{V}$ such that $B_{r_i /4}(x_i)$ are disjoint to each other and 
\[
B_{R/2 } \setminus D_{k} \subset \cup_i B_{r_i }(x_i).
\]
Moreover, by the choice of $r_x$, we also have $B_{r_i/ 4} (x_i) \cap D_{k}  = \emptyset  $ and henceforth,
\[
 \biggl( \bigcup_i  B_{r_i /4}  \cap B_{R/2} \biggr)  \subset B_{R/2} \setminus D_{k}
\]

Now, we compute 
\[
\begin{split}
\nu [B_{R/2} \setminus D_k] &  \leq \nu [\cup_i B_{r_i}(x_i) ] \leq \sum_i \nu [B_{r_i}(x_i)]   \leq \sum_i \frac{ 1 }{\mu} \nu [ D_{k+1} \cap B_{r_i /4}(x_i)]\\
 & = \frac{ 1 }{\mu  }  \nu [\cup_i ( D_{k+1} \cap B_{r_i/4}(x_i) ) ]  = \frac{ 1 }{\mu} \nu [ D_{k+1}\cap  (\cup_i B_{r_i/4}(x_i))] \\
& \leq \frac{1}{\mu  } \nu [ D_{k+1} \cap (B_{R/2} \setminus D_{k})]. 
\end{split}
\]
Here, at the first equality in the second line, we have used the fact that $B_{r_i/4}$ are disjoint. Now we have proved the claim. 

Recall the definition of $\tilde{\lambda} (t) = \tilde{\lambda}_{B_{R/2}} (t)$ (Eq.\ref{distribution}), we immediately obtain from the claim that 
\[
\tilde{\lambda} (M_{R}^{k+1})  \leq (1 -\mu ) \tilde{\lambda} (M_{R}^k).
\]
The desired estimate follows by inductively apply this inequality.
\end{proof}

\begin{proof}[Proof of Thm.\ref{Harnack sup finite}]
Denote $\eta = \eta_{K,N, 2R}$. Replacing $u$ by 
\[ 
\tilde{u}=\frac{u}{e^{1/p_0}(  \inf_{B_{R/2}} u + \I_{K, N} (f , B_{2R} ,\eta ) ) },
\]
we have
\[
\Delta \tilde{u} \leq  [ e^{1/p_0}(  \inf_{B_{R/2}} u + \I_{K, N} (f ; B_{2R}; \eta ) )]^{-1} f = \tilde{f}.
\] 
Then, recall ii) of Lem.\ref{constants estimate}, $\tilde{u}$ satisfies
\[
\inf_{B_{R/2}} u \leq 1 \quad \I_{K,N} (\tilde{f} ,B_{2R}, \eta) \leq \frac{1}{e^{1/p_0}} \leq \delta_0
\]

Hence, by this normalziation, may assume
\[
\inf_{B_{R/2}} u \leq 1 \quad \;   \I_{K, N} (f ; B_{2R}; \eta )\leq \delta_0.
\] 
Therefore, it suffices to bound averaged-$L^{p_0}$-norm of $u$ by $e^{1/p_0}$.

Recall i) of Lem.\ref{constants estimate}, $p_0$ satisfies
\[
1 + (M^{p_0} - 1 )\sum_{k=0}^{\infty} M^{p_0 k} \tilde{\lambda} (M^k) \leq e
\]
Applying Lem.\ref{decay}, Lem.\ref{Lp} along with the above identity, we 
have
\[
 (\fint_{B_{R/2} } u^{p_0})^{1/p_0} \leq e^{1/p_0}.
\]
The desired estimate is then obtained by re-normalizing $u$ back and recall the choice of constant $C_0$.
\end{proof}

\bigskip

\section{Proof of Harnack Inequalities III}
\label{soln}
In this section, we give the proof of Thm.\ref{Harnack sub finite} and Theorem \ref{Harnack soln finite}. 

Recall the constant $p_0$ in the statement of Thm.\ref{Harnack sup finite} (see Lem.\ref{constants estimate}). Recall the constant $\delta_0, M, \mu, p_1$ from \S \ref{notations}. Again, they are chosen w.r.t to the large ball $B_{2R}$. We denote $\eta_{K,N,2R}$ (Eq.\ref{C eta}) by $\eta$ in the rest of this section.

The key part is the following lemma. 
\begin{lem}
\label{soln sup}
Under the assumption of Thm.\ref{Harnack sub finite}. Suppose additionally 
\[
\biggl( \fint_{B_{R}} u^{p_0} \biggr)^{1/p_0}\leq 1, \quad \mathcal{I}_{K,N} (f , B_{2R}, \eta) \leq \delta_0.
\]
and
\[
\beta:= u (x_0 )  > M , \quad B_{r_0} (x_0) \subset B_{R} \text{ with } r_0 = RC_3 \beta^{-p_1} 
\]
Then,
\[
\sup_{B_{r_0} (x_0)} u \geq \beta (1 + 1/M).
\]
\end{lem}

\begin{proof}

Argue by contradiction. Suppose
\[
\sup_{B_{r_0} (x_0)} u < \beta (1+ 1/M)
\]
Consider the function
\[
w = \frac{\beta (1 + 1/M) - u  }{\beta /M}.
\]
on $B_{r'} (y_0)$ with 
\[
r' = r_0/4, \quad \rho (y_0, x_0) =r_0/8.
\] 
Note that $w$ satisfies
\[
\Delta_{\nu} w \leq \frac{-f}{\beta/M} \leq \abs{f} \text{ in } B_{r'} (x_0)
\]
and 
\[
w \geq  0 \text{ on } B_{r'} (y_0), \quad \inf_{B_{r'/2} (y_0)} w \leq w(x_0) \leq 1
\]

Thus, we may apply the  Lem.\ref{local growth} to obtain
\begin{equation}
\label{measure growth}
\mu \nu[B_{r'}(y_0) ] \leq \nu [\{w \leq M \} \cap B_{r'/18} (y_0)]  
\end{equation}

Observe that  
\[
w \leq M \Rightarrow u \geq \beta /M.
\]
Hence, (Eq.\ref{measure growth}) along with the Chebyshev's inequality implies
\[
\beta\leq \frac{  M }{ \mu^{1/p_0} } \biggl( \fint_{B_{r'} (y_0)} u^{p_0}  \biggr)^{1/p_0} 
\leq  \frac{  M }{ \mu^{1/p_0} }  \biggl( \frac{\nu [B_{R}]}{\nu[B_{r' }(y_0)]} \biggr)^{1/p_0}.
\]

Now, recall the choice of $r' =r_0/4$, $p_1$ (Eq.\ref{C pr} in \S \ref{notations}) and apply the doubling estimate (Eq.\ref{doubling estimate}) to $B_{r'} (y_0) \subset B_{R}(y_0) \subset B_{2R}$, we obtain 
\[
\beta \leq\frac{  M}{ \mu} \db_{2R}\biggl( \frac{R}{r'} \biggr)^{N\eta/p_0} \leq  \frac{  M }{ \mu^{1/p_0} } \db_{2R}  \biggl( \frac{1}{C_3} \biggr)^{N\eta/p_0} \beta 
\]

Recalling the choice of $C_3$ (Eq.\ref{C C3}, Eq.\ref{C3 estimate}), then the above inequality implies $\beta < \beta$, which is impossible.
\end{proof}

\medskip

Both Thm.\ref{Harnack sub finite} and Thm.\ref{Harnack soln finite} can be deduced from Lem.\ref{soln sup}. Recall the constants $M , p_1, C_3, \mu$ from \S \ref{notations}.

\begin{proof}[Proof of Thm.\ref{Harnack sub finite}]
We first prove Thm.\ref{Harnack sub finite} with $p = p_0$. By replacing $u$ by
\[
\tilde{u} =  \biggl\{ \biggl( \fint_{B_{R}} u^{p_0} \biggr)^{1/p_0} + \mathcal{I}_{K,N} (f ; B_{2R}) /\delta_0\biggr\}^{-1} u,
\]
Similar to the proof of Thm.\ref{Harnack sup finite}, we may assume $u$ satisfies the hypothesis in Lem.\ref{soln sup}. Thus, it suffices to bound $\sup u$ by the constant $C_1 (p_0)$ (Eq.\ref{C C1p0}) given in \S \ref{notations}.

Argue by contradiction. Suppose there exists $x_0 \in B_{R/2}$ such that 
\begin{equation}
\label{contradiction}
\beta := u (x_0) > C_1 (p_0).
\end{equation}
Consider the sequence $x_k$ defined as follows: choose inductively $x_{k+1}$ such that
\[
u (x_{k+1} ) = \sup_{B_{r_k} (x_k) } u , \quad r_k = RC_3 u(x_k)^{-p_{1}} 
\]

Consider the sum,
\begin{equation}
\label{cond1}
 RC_{3} \sum_{k=0}^{\infty} \frac{1}{\beta^{p_1} (1+ 1/M)^{kp_1}} 
\end{equation}

The hypothesis on $\beta$ (Eq.\ref{contradiction}) and the choice of $C_1 (p_0)$ together implies
\[
 RC_{3} \sum_{k=0}^{\infty} \frac{1}{\beta^{p_1} (1+ 1/M)^{kp_1}}\leq R/3.
\]
Hence, we may inductively apply the Lem.\ref{soln sup} to obtain 
\[
u(x_k) \geq C_1 (p_0) (1 + 1/M)^k , \quad r_k < \frac{RC_3}{C_1^{p_1} (p_0) (1 + 1/M)^{p_1 k}}.
\]
and
\[
\sum_{k} r_k \leq \frac{R}{3}.
\]  

However, this implies 
\[
x_k \in B_{r_k} (x_k ) \subset B_{R} ,\quad \forall k \geq 0
\]
and $u (x_k)$ tend to $\infty$ in $B_{R}$. This contradicts the fact that $u$ is continuous in $\overline{B}_R$ and hence prove the case $p = p_0$.

For $p > p_0$, one may take $C_1 (p) = C_1 (p_0)$ and the desired inequality follows from standard $L^p$ theory. For $p < p_0$, one may apply a standard interpolation argument (eg. Ch.4 of \cite{HL}). In that argument, one shall need the factor $\eta = \eta_{K, N, 2R}$ given by the doubling property and get 
\[
C_1 (p) = \tilde{C}_1 (p, N, \eta) C_1 (p_0) .
\]
Note $\eta$ only depends on $\sqrt{K}R$.
\end{proof}

\medskip

\begin{proof}[Proof of Thm.\ref{Harnack soln finite}]
By replacing $u$ by
\[
\frac{u }{ C_0 (\inf_{B_{R/2}} u + \mathcal{I}_{K,N} (f ; B_{R})) },
\]
Similar to the proof of Thm.\ref{Harnack sup finite}, we may assume
\[
\inf_{B_{R/2}} u + \mathcal{I}_{K,N} (f ; B_{R}) \leq \frac{1}{C_0} , \quad  \mathcal{I}_{K,N} (f ; B_{R})  \leq \delta_0
\]

By Thm.\ref{Harnack sup finite}, we obtain
\[
\biggl( \fint_{B_{R/2}} u^{1/p_0}   \biggr)^{1/p_0}  \leq 1
\] 

Using Lem.\ref{soln sup} and following the exactly same argument as in the proof of Thm.\ref{Harnack sub finite}, we obtain 
\[
\sup_{B_{R/2} }u \leq C_1 (p_0) =C_2
\]

The desired estimate is achieved by re-scaling $u$ back.
\end{proof}

\bigskip

\section{Characterization of Ricci Lower Bound}
In this section, we formulate some questions regarding characterizing Ricci lower bound by the Harnack inequality.

\smallskip

We first define the Harnack functional which forms the foundation of our discussion.
 \begin{defn}
\label{Harnack functional}
Let $(\Mfd, g)$ be a Riemannian manifold. The Harnack functional associate to the metric $g$ is defined by 
\begin{equation}
\mathscr{H}_{g} (r,x):=  \sup_{u} \biggl\{ \frac{\sup_{B^g_{r/2} (x)} u}{ \inf_{B^g_{r/2}(x)} u } : \Delta_{g} u  =0, u > 0 \text{ in } B^g_{r} (x)    \biggr\}  
\end{equation}
where $r > 0$ and $x \in \Mfd$.
\end{defn}

\begin{rem}
The superscript and superscript $g$ is to emphases that $B^{g}$ is the geodesic ball w.r.t to the metric $g$ and $\Delta_{g}$ is the Laplacian w.r.t to $g$.
\end{rem}

\begin{rem}
From the Harnack inequality (Thm.\ref{Harnack soln finite}), we see $\Har$ is well-defined. It is the best Harnack constant \textit{for the given Riemannian manifold}. Note here, the Riemannian manifold is fixed in priori. There are several Harnack inequalities with sharp constant exists in the lecture. However, the sharp there means for a given $K$, there is a manifold and a harmonic function to realize the inequality. So they cannot be used to calculate $\Har$ in explicit.
\end{rem}

Denote $\lambda_1 [\Ric_g] (x)$ the smallest eigenvalue of $\Ric_g$ at  $x$. Our idea is to consider asymptotic behaviors of $\Har_{g} (r,x)$ for $r$ small. We believe $\Har_{g}$'s asymptotic behavior should characterize $\lambda_1 [\Ric_g] $. More precisely, we believe that fix given point $x$, expansion of $\Har_{g} (r,x)$ near $r =0$ characterizing the smallest eigenvalue of $\Ric$ at the point. The concrete questions are follows.

\textbf{Question 1}: Given a smooth Riemannian Manifold $(\Mfd, g)$. Fix a $x \in \Mfd$, does $\Har_{g} (r,x)$ have a series expansion near $r = 0$? Suppose this is the case, 
\[
\Har_{g} (r,x) = a_0  + a_1 r + a_2 r^2 + O(r^3), \quad r\rightarrow 0
\]
Then, whether the following statement holds: 

i) Is $a_1$ always zero for all $(\Mfd,  g)$ and for any $x \in \Mfd$; 

ii) Is there a universal constant $\mathbf{h}$ only depends on dimension, does Not vary according to $(\Mfd, g)$ nor point $x$, such that
\[
\mathbf{h}\;  a_2 (x, g)  = \lambda_1 [ \Ric_{g} ] (x) .
\]

iii) Can $\Har_{g}$ be extended as a even function w.r.t to $r$?

The above questions might be too surprising from some aspects. However, we do believe they have affirmative answers, at least for i) and ii). The following proposition gives some simple evidence of our guess. 

\begin{prop}
i) On $\R^2$ with standard Euclidean metric $g_0$, $\Har_{g_0} (r,x)$ is constant in both argument and
\[
\Har_{g_0} = 9
\]
\smallskip

ii) On $\mathbb{S}^2$ the sphere with metric $g_{K}$ whose Ricci is constant $K\geq 0$ (see. Eq.{complex Ricci}), $\Har_{g_{K}} (r,x)$ is constant in $x$ and has an even extension w.r.t $r$. Moreover, it has the following expansion
\[
\Har_{g_{K}} (r, x) = 9 - 3 K d^2 + \frac{3}{8}K^2 d^4 - \frac{11}{480} K^3d^6  + O(d^7), \quad d\rightarrow 0
\]
\smallskip 

iii) On $\mathbb{H}^2$ the Hyperbolic plane with metric $g_{-K}$ whose Ricci is constant $ -K$ with $K\geq 0$, $\Har_{g_{-K}} (r,x)$ is constant in $x$  and has an even extension w.r.t $r$. Moreover, it has the following expansion
\[
\Har_{g_{-K}} (r, x) = 9 - 3 (-K) d^2 + \frac{3}{8} (-K)^2 d^4 - \frac{11}{480} (-K)^3d^6  + O(d^7), \quad d\rightarrow 0
\]
\end{prop}

\begin{proof}
In $\R^2$, Harmonic functions are invariant under the scaling:$
v (x) := u (r x )$ and the translation $v(x) = u (x + y_0)$. Hence clearly from the definition. $\Har (r,x)$ is a constant independent of $x$ and $r$. It has a trivial expansion. and its $a_2 = 0$ is the curvature of $\R^2$.

By the Poisson integral formula
\[
u (x) := \frac{1 - \abs{x}^2}{ 2\pi   } \int_{0}^{2\pi} \frac{u (\omega) d\omega}{\abs{ x -e^{i\omega} }^2}
\]
one may give a \textit{sharp} estimate that
\begin{equation}
\label{sharp Harnack}
\frac{ \sup_{B_{\theta r} (x)} u }{ \inf_{B_{\theta r} (x)} u } \leq \frac{(1 + \theta)^2}{(1- \theta )^2}
\end{equation}
In particular take $\theta = 1/2$. we proves i).

\smallskip

To consider $\mathbb{S}^2$ and $\mathbb{H}^2$, it is convenient to consider the conformal formulation and use complex coordinates. 

Recall $\mathbb{S}^2 $ can be represented as
\[
(\C, g_{K} = \sigma^2 (z) dz \wedge d\bar{z}), \quad \sigma_{K} (z):= \frac{\sqrt{2}}{(1+ K\abs{ z}^2)}
\]
whose Ricci curvature is
\begin{equation}
\label{complex Ricci}
\Ric =  -\partial \bar{\partial} \log  \sigma^2 (t) = \frac{2K}{(1+ K\abs{z}^2)^2}= K \sigma^2 
\end{equation}

By the homogeneity of $\mathbb{S}^2$, it suffices to consider $\Har_{g_K} (r, 0)$.

Recall the relation between the geodesic $B^{g_K}_d (0)$ and Euclidean ball $B_{r} (0)$ (in this proof balls without superscript means Euclidean ball)
\[
B^{g_K}_d (0) = B_r (0), \text{ where } r = \tan ( \sqrt{K} d/\sqrt{2})
\]
This relation follows easily from integrate $\sigma (t)$ along rays in $\R^2$ and observes that rays in $\R^2$ passing $0$ is geodesic on $\mathbb{S}^2$. 

Also recall the Laplacian operator $\Delta_{g_K}$, w.r.t to the conform metric, is of the following form
\[
\Delta_{g_K} = \frac{1}{\sigma^2 } \Delta ( \text{ the standard Laplacian }).
\]

Combine these two relations, we see positive harmonic functions on $B^{g_K}_{d} (0)$ are positive harmonic functions on $B_{r} (0)$ with $r = \tan (d/\sqrt{2})$ and vice versa. Hence, 
\[
\Har_{g_{K}} (d, 0) = \sup_{u} \biggl\{ \frac{\sup_{B_{r'} (x)} u}{ \inf_{B_{r'}(x)} u } : \Delta u  =0, u > 0 \text{ in } B_{r} (x)    \biggr\}  
\] 
where $r' = \tan (d/(2\sqrt{2}))$ and $r = \tan(d/\sqrt{2})$. Therefore apply the sharp Harnack estimate with
\[
\theta := \frac{r'}{r}  = \frac{\tan (\sqrt{K}d/(2\sqrt{2})) }{ \tan(\sqrt{K}d/\sqrt{2})}
\]
We obtain
\[
\Har_{g_K} (d, 0 ) =\biggl[1 + 2 \cos \biggl(\frac{\sqrt{K} d}{\sqrt{2}}\biggr) \biggr]^2
\]
which is clearly even; By direct calculation, we obtain the desired expansion.

\smallskip

The argument for Hyperbolic $2$-pane is exactly same. This time
\[
\mathbb{H}= (\mathbf{D}=\{\abs{z}< 1\} , \quad g_{-K}  = \frac{2}{(1 - K\abs{z}^2)^2} dz\wedge d\bar{z}.
\]
The correspondence between geodesic distance $d$ from $0$ and Euclidean distance $r$ is
\[
 \sqrt{2K} d = \log \frac{1+r}{1-r} .
\]
With the exactly same way of calculation, we obtain
\[
\Har_{g_{-K}} (d, 0) = \biggl[1 + 2 \cosh \biggl(\frac{\sqrt{K} d}{\sqrt{2}}\biggr) \biggr]^2
\]
iii) then follows.
\end{proof}

\begin{rem}
The above calculation method is also work for the metric $g = (1 + \abs{z}^2)^2 dz \wedge d\bar{z}$ and gives desired expansion of $\Har_{g} (r)$ at $z=0$. Note this metric is not constant curvature any more. 
\end{rem}

Certainly, the above examples are very special; and the method to check that their Harnack functional have desired properties is very limited. However, they give some positive evidence for our guess.


Intuitively, we think the first order coefficient $a_1$ should always vanishes. Indeed, letting $r$ tending to zero is equivalent to consider an infinitesimal ball. But the metric geometry of a Riemannian manifold does not differ from Euclidean space in first order approximation. Thus, the Harnack inequality for this infinitesimal ball should agree with that on Euclidean space up to first order.



\medskip

It is clear that all the formulations in this section have their corresponding part for BE-Ricci curvature. However, in the presence of non-trivial reference measure, one also need to take the effective dimension into count. We believe the Harnack functional defined similar to Defn.\ref{Harnack functional} shall relate to the product $\sqrt{KN}$. It should related to the optimal effective dimension of the Riemannian metric measure space 

It seems also of interests to consider the asymptotic behavior of $\Har_{g} (r)$ for $r$ tending to $\infty$ on non-compact spaces. This asymptotic behavior of $\Har_{g}$ should indicate some information on the way $\lambda_1 [\Ric_g ]$ distributed on the manifold. However, this issue is quite subtle and we are currently of no idea about it. 

We believe the study of Harnack functional will have applications in the study of geometric flow, Gromov-Hausdroff convergence and Alexandrov spaces.

\bigskip

\section{Appendix: Fully-Nonlinear Uniform Elliptic Equations}
In this appendix, we briefly explain how the proof in this paper might be extended to cover fully-nonlinear uniform elliptic equations on Riemannian manifolds ($\nu = \vol_{g}$). 

Recall from the standard theory of fully-nonlinear PDEs, it suffices to prove the Harnack inequalities for the Pucci-extremal operator
\begin{defn}
Let $ \theta \geq 1$ be a constant and $u \in C^2$. The Pucci extremal operator, denoted by $\M_{\theta}^{+}$, is defined by 
\[
\begin{split}
& \M^{-}_{\theta} [u] (x) := \M^{-}_{\theta} ( H ):= \sum_{\lambda_i \geq 0} \lambda_i ( H ) + \theta \sum_{\lambda_i \leq 0} \lambda_i (H)  \\
&  \M^{+}_{\theta} [u] (x) :=\M^{-}_{\theta} (H ):= \sum_{\lambda_i < 0} \lambda_i (H) + \theta \sum_{\lambda_i \geq 0} \lambda_i (H) 
\end{split}
\]
where $H = \nabla^2 u (x)$ and $\lambda_i (H)$ denote eigenvalues of $H$.
\end{defn}
 
The next lemma demonstrate the generality and the extremity of the Pucci operator (see\cite{CC}).
\begin{lem}
Let $\Sym T\Mfd$ be the bundle of $g$-self-adjoint operators on $T\Mfd$. Let $H \in \Sym T\Mfd$ be a section
\[
\begin{split}
& \M^{-}_{\theta} ( H ):=\inf \{ \tr \bigl[ A \cdot H \bigr],  A \in \Sym T\Mfd, Id \leq A \leq \theta Id \} \\
&  \M^{+}_{\theta} (H ):= \sup \{\tr \bigl[A \cdot H \bigr] ,  A \in \Sym T\Mfd, Id \leq A \leq \theta Id \}
\end{split}
\]
\end{lem}

For the purpose of investigating Pucci operator, the following quantity is very convenient (see.\cite{Kim})
\begin{equation}
\label{error}
\E_{\theta} (r):= \sup_{(x,y)}\{ \M^{+}_{\theta} [\nabla^2(\frac{1}{2}\rho^2_y)] (x)-  \tr [\nabla^2(\frac{1}{2} \rho^2_{y}) ] (x) , ], \; \rho (x, y) \leq r \}
\end{equation}

The following lemma explains some property of this quantity
\begin{lem}
\label{error estimate}
Let $(M, g)$ be a complete Riemannian manifold. Then the following statement is true

i) If $\nabla^2 (\frac{1}{2} \rho^2_y) (x) $ is non-negative definite for all $x, y$ with $x, y\in B_{R}$, then
\[
\E_{\theta} (2R) \leq (\theta - 1) \bigl(1 + (n-1) \tc (\omega_{K, n}R) \bigr)
\]
where $K$ is the Ricci lower bounded in $B_{R}$. In particular, this holds for any ball $B_{R}$ with $R \leq R_s$, where $R_s$ depends on sectional curvature upper bounded in $B_{2R}$.
 
\smallskip
ii) Suppose in $B_{R}$, sectional curvatures are bounded below by $- K_{s} $ with $K_s \geq 0$. Then 
\[
\E_{\theta} (2R) \leq (\theta - 1) \bigl(1 + (n-1) \tc (\sqrt{K_{s}/n}R) \bigr).
\]
\end{lem}

\begin{proof}
The first part of i) follows immediately from the definition of $\E_{\theta} (R)$, Ricci comparison. The second part of i) follows by applying sectional curvature comparison against spheres.

ii) follows from standard sectional curvature comparison.
\end{proof}

\medskip

The proof of Thm.\ref{Harnack sup finite} -- Thm.\ref{Harnack soln finite} can be easily extended to prove the following Harnack inequalities for Pucci extremal operator. 

\begin{thm}
\label{fully nonlinear}
Let $(\Mfd, g, \nu)$ be a Complete Smooth Riemannian Metric-Measure Space. Let $K \geq 0$ be a constant. Let $u \in C^2 (B_{2R}) \cap C(\overline{B}_R)$ and $f \in C (B_{2R})$. 
Suppose
\[
\Ric |_{B_{2R}} \geq -K
\]


Then, by replacing $\sqrt{K}$ in the dependence of $p_0, C_1, C_2$ (not including $\eta$) with
\[
R\sqrt{K} + \E_{\theta} (2R)
\]
the following statements holds:

i) $\M^{-}_{\theta} [u] \leq f$ and $u \geq 0$ in $B_{2R}$ implies the inequality (Eq.\ref{sup inequality}); 

\smallskip

ii) $\M^{+}_{\theta} [u] \geq f$ in $B_{2R}$ implies the inequality (Eq.\ref{sub inequality}); 

\smallskip

iii) $\M^{-}_{\theta} [u] \leq \abs{f}$ and $\M^{+}[u] \geq f$ and $u \geq 0$ in $B_{2R}$ implies the inequality (Eq.\ref{soln inequality});
\end{thm}

\begin{rem}
In case $Sec |_{B_{2R}} = 0$, Thm.\ref{fully nonlinear} along with the Lem.\ref{error estimate} recovers the Harnack inequality proved in \cite{Cab}. Thm.\ref{fully nonlinear} clearly fulfill the claim made in \cite{Cab} regarding the space with sectional curvature lower bound. Indeed, Thm.\ref{fully nonlinear} can be viewed as an extension of \cite{Kim}.
\end{rem}


\begin{rem}
Thm.\ref{fully nonlinear} and Lem.\ref{error estimate} together implies that when the domain $B_{2R}$ is small enough (depend only on sectional curvature lower bound), then the estimate only requires local Ricci lower bound. 
\end{rem}

\begin{rem}
Unlike the divergence case, in case $\theta \neq 1$, we think the dependence on $\E_{\theta} (r)$ (hence sectional curvature) can NOT be replace by Ricci curvature. This can be seen in the following way. Fix some $x \in \Mfd$, if Hessian $H (x) := \nabla^2(\frac{1}{2}\rho^2_{y} (x))$ has a negative eigenvalue, then by choosing $\theta$ large, 
\[
\M^{+}_{\theta} (H) \sim  \theta \lambda_{n} (H), \M^{-}_{\theta} (H) \sim \theta \lambda_1 (H)
 \]
where $\lambda_1 , \lambda_n$ is the smallest and largest eigenvalue of $H$. Thus, $\M^{+}_{\theta}, \M^{-}_{\theta}$ are affected by the extremal eigenvalues of $H$. While in divergence case, the operator ($Id\leq A \leq \theta A$)
\[
\int g ( A \nabla \frac{\rho^2_y}{2}, \nabla \varphi) \sim \theta \rho_{y} \abs{\nabla \varphi}  , \quad \varphi \text{ test function}
\]  
is only affected by $\rho$. From this comparison, the dependence on $\E _{\theta} (R)$ seems necessary. 
\end{rem}

\medskip

\begin{proof}[Proof of Thm \ref{fully nonlinear}]
Recall the proof in \S \ref{sup} -- \S\ref{soln} only relies on the Lem.\ref{local growth}, hence it suffices to prove Lem.\ref{local growth} for $\M^{-}_{\theta}$ with $R\sqrt{K}$ in the constant $\mu, M$ replaced by
\[
R\sqrt{K} + \E_{\theta} (2R).
\]

Recall the proof of Lem.\ref{local growth}, we then see it suffices to control $\Delta w$ from above on contact set $A = A (a, B_{r/6} (y_0)/B_{r} (x_0), w)$ (recall $w = u + \psi$) by $\M^{-}_{\theta} u \leq f$ and $R\sqrt{K} + \E_{\theta}$. This can be done by the following simple calculation

At $x \in A$, by contact condition, we have
\[
\nabla^2 w (x) \geq - a \nabla^2 ( \frac{1}{2}\rho_{y}^2 ) (x),
\] 
Denote $ S = \nabla^2_{\nu} w (x)$ and $H =H_{y}^{\nu}(x) $. We then have $S + a H \geq 0$.

The result follows from the following calculation.
\[
\begin{split}
\Delta u (x) & = \tr [ S + a H - aH ]  = \tr [ S+ aH] - a \tr [ H] \\
&  \leq \M_{\theta}^{-} (S + aH) - a\tr [H] \leq  \M^{-}_{\theta} (S) +  a \E_{\theta} (2r)
\end{split}
\]

While $S = \nabla^2 u + \nabla^2 \psi$, by the elementary inequality regarding Pucci operator, we have
\[
\M^{-}_{\theta} (S) \leq \M^-_{\theta} [u] (x) + \M^{+}[\psi] (x)
\]

Following the construction of the barrier $\psi$ (Lem.\ref{barrier 2}), one can easily see
\[
r^2 \M^{+} [\psi]\leq - \alpha (\alpha +2) + a \E_{\theta} (2r).
\]
 
Hence by replacing $\sqrt{K}R$ with $\sqrt{K}R + \E_{\theta} (R)$  in the expression of $\alpha$, $\Delta u$ is controlled on the contact set in the desired way.

The rest of the proof of the Thm.\ref{fully nonlinear} follows line by line from our previous proof \S \ref{sup} --\S \ref{soln}.
\end{proof}

$\\$
{\bf Acknowledgements:} The first named author would like to thank Prof. Duong Hong Phong for his penetrating remarks and advice. The second named author would like to thank Prof. Pengfei Guan for his helpful discussions and constant support. Both authors want to thank Prof. Ovidiu Savin for his helpful suggestions in PDEs and Prof. Ovidiu Munteanu in differential geometry. We also want to express gratitude to Prof. Diego Maldonado for pointing out an error in the first version of our paper. 
$\\$

\bibliographystyle{plain}
\bibliography{Harnack3}

\end{document}